%% file: main.tex
\documentclass[10pt]{amsart}
\usepackage[utf8]{inputenc}
\usepackage{xcolor,amsmath,amssymb,amsfonts,hyperref,mathtools,tikz-cd,array,orcidlink}

\usepackage{amsthm}
\usepackage{multicol}
\usepackage{enumitem}
\usepackage[misc]{ifsym}
\usepackage{algorithm}
\usepackage{algpseudocode}

\usepackage{tikz-cd}
\usepackage{comment}
\usepackage[capitalise]{cleveref}
\usepackage{csquotes}
\allowdisplaybreaks

\makeatletter
\algnewcommand{\linecomment}[1]{\Statex \hskip\ALG@thistlm \textcolor{blue}{/*#1*/}}
\makeatother

\setlist[itemize]{leftmargin=7.5mm}

\newtheorem{theorem}{Theorem}[section]
\newtheorem*{theorem*}{Main Theorem}
\newtheorem{proposition}[theorem]{Proposition}
\newtheorem{corollary}[theorem]{Corollary}
\newtheorem{lemma}[theorem]{Lemma}

\theoremstyle{definition}
\newtheorem{definition}[theorem]{Definition}
\newtheorem{example}[theorem]{Example}

\newtheorem{heuristic}[theorem]{Heuristic}
\newtheorem{remark}[theorem]{Remark}

\newcommand{\orcid}[1]{\href{https://orcid.org/#1}{\textcolor[HTML]{A6CE39}{\aiOrcid}}}

\newcommand{\C}{\mathcal{C}}

\newcommand{\A}{\mathcal{A}}
\renewcommand{\L}{\mathcal{L}}
\newcommand{\Am}{\mathsf{A}}

\newcommand{\E}{\mathcal{E}}
\renewcommand{\O}{\mathcal{O}}

\renewcommand{\P}{\mathcal{P}}

\newcommand{\Jac}{\text{Jac}}

\newcommand{\FF}{\mathbb{F}}
\newcommand{\Fq}{\mathbb{F}_{q}}
\newcommand{\Zq}{\mathbb{Z}_{q}}
\newcommand{\Fp}{\mathbb{F}_{p}}
\newcommand{\QQ}{\mathbb{Q}}
\newcommand{\ZZ}{\mathbb{Z}}
\let\Z\ZZ
\let\F\FF

\newcommand{\NN}{\mathbb{N}}
\newcommand{\CC}{\mathbb{C}}
\newcommand{\PP}{\mathbb{P}}
\newcommand{\End}{\mathop{End}}

\newcommand{\Spec}{\mathrm{Spec}}

\renewcommand{\char}{\textrm{char}}
\newcommand{\disc}{\textrm{disc}}
\newcommand{\Otilde}{\tilde{O}}

\newcommand{\iso}{\ensuremath{\simeq}}      
\DeclareMathOperator{\Ker}{Ker}
\newcommand{\const}{\mathfrak{O}}
\newcommand{\corresp}{\overline{{\Phi}_{\ell}}}
\newcommand{\correspbis}{\overline{{\Phi}_{2}}}
\newcommand{\hodge}{\mathfrak{h}}
\let\tens\otimes

\newcounter{savealgorithm}

\providecommand{\stackTag}[1]{
\href{https://stacks.math.columbia.edu/tag/#1}{Tag~#1}}
\providecommand{\stackcite}[1]{
\cite[\stackTag{#1}]{stacks-project}}

\definecolor{lightgreen}{rgb}{0.86, 0.93, 0.78}
\definecolor{bordergreen}{rgb}{0.55, 0.76, 0.74}
\definecolor{lightblue}{rgb}{0.70, 0.90, 0.99}
\definecolor{borderblue}{rgb}{0.01, 0.66, 0.96}
\definecolor{lightamber}{rgb}{1, 0.93, 0.70}
\definecolor{borderamber}{rgb}{1, 0.76, 0.03}
\definecolor{lightcolor4}{rgb}{ 0.93, 0.70, 1}
\definecolor{bordercolor4}{rgb}{0.76, 0.03, 1}
\definecolor{lightcolor5}{rgb}{0.78,0.86,0.93}
\definecolor{bordercolor5}{rgb}{0.74,0.55,0.76}
\RequirePackage{totcount}
\RequirePackage{color}
\RequirePackage[colorinlistoftodos]{todonotes}

\newtotcounter{notecount}
\title{Computing modular polynomials by deformation}

\author{Sabrina Kunzweiler, Damien Robert}
\address{Inria Bordeaux, Institut de Math\'ematiques de Bordeaux}

\begin{document}
	
	\begin{abstract}
		We present an unconditional CRT algorithm to compute the modular polynomial $\Phi_\ell(X,Y)$ in quasi-linear time. The main ingredients of our algorithm are: the embedding of $\ell$-isogenies in smooth-degree isogenies in higher dimension, and the computation of $m$-th order deformations of isogenies. We provide a proof-of-concept implementation of a heuristic version of the algorithm demonstrating the practicality of our approach. 
		
		Our algorithm can also be used to compute the reduction of $\Phi_{\ell}$
		modulo~$p$ in quasi-linear time (with respect to $\ell$) $\Otilde(\ell^2
		(\log p + \log \ell)^{\const})$.
	\end{abstract}
	
	\maketitle
	
	\input{introduction}
	\input{deformations}

	\input{explicit-deformations}

	\input{modular-Fp}

\input{modular-general}
	\input{implementation}

	\bibliographystyle{plain}
	\bibliography{ref.bib}
	
	\newpage
	\appendix
	\input{variants}
\end{document}

%% file: introduction.tex
\section{Introduction}

The modular polynomial $\Phi_{\ell}(X,Y) \in \Z[X,Y]$ parametrizes
$\ell$-isogenies between elliptic curves. It is a fundamental object for
building isogenies and exploring isogeny graphs, and is used in many
applications, notably in the Schoof–Elkies–Atkin (SEA) point counting algorithm \cite{schoof1995counting, elkies1991explicit, atkin1988number}.
We will assume $\ell$ to be prime for simplicity, the extension to the
general case is not hard.
It is well known that $\Phi_{\ell}$ has degree $\ell+1$ in each
variable with height $O(\ell \log \ell)$, so the modular polynomial has size
$O(\ell^3 \log \ell)$. More explicitly, it is shown in \cite{broker2010explicit} that 
\begin{equation} \label{eq:bound}
	\log |a_{i,j}| \leq 6 \ell \log(\ell) + 16 \ell + \min{(2\ell,14\sqrt{\ell}{\log{\ell}})}, 
\end{equation}
where $\Phi_\ell = \sum_{i,j} a_{i,j}X^iY^j$. Refined bounds can be found in \cite{breuer2023explicit}.

\subsection{Algorithms for computing modular polynomials}

The first quasi-linear algorithm to compute $\Phi_{\ell}$ is due to Enge in
\cite{enge2009computing} and uses analytic methods. Namely, it uses an
evaluation-interpolation approach: fix $\ell+1$ nice period matrices $\tau_i$ in
the complex upper half plane, and evaluate
$\Phi_{\ell}(j(\tau_i), Y) = \prod_{j=1}^{\ell+1} (Y - j(\tau_{i,j})) =
Y^{\ell+1}+\sum_{j=0}^{\ell} c_j(j(\tau_i)) Y^j$,
where the $\tau_{i,j}$ are period matrices of the $\ell+1$ elliptic curves
isogenous to $E_{\tau_i}$.
Each coefficient $c_j$, seen as polynomial in the variable $X$, can then
be recovered by interpolation.

Since the modular polynomial has large height, the analytic method needs to
work with large precision.
To achieve quasi-linearity, the $j$-invariants thus need to be evaluated
in quasi-linear time in the precision. Due to the use of large precision floating point arithmetic, Enge's algorithm is only heuristic: the assumption is that the loss of precision in the floating point arithmetic does not explode.
This has been made rigorous in recent work by Kieffer \cite{kieffer2022certified,kieffer2020evaluating}.
In \cite{kieffer2020evaluating}, Kieffer focuses on the rigorous evaluation of dimension 2 theta
functions; but the same arguments work (and are easier) to compute the
modular polynomials in dimension 1.

In practice, the analytic algorithm works well; due to the quasi-linear
complexity,
the blocking factor to compute large modular polynomials are their size,
hence the available memory, rather than the computing power.
In \cite{broker2012modular}, Bröker, Lauter and Sutherland introduce a quasi-linear algorithm based on the Chinese remainder theorem (CRT) to compute $\Phi_{\ell}$. Besides being purely algebraic, a great advantage is that thanks to the explicit CRT algorithm \cite{bernstein2007modular}, the reduced modular polynomial $\Phi_{\ell}(X,Y) \mod m$ ($m$ a large integer) can be computed by using a memory of only $O(\ell^2 \log (m \ell))$ rather than $O(\ell^3 \log \ell)$.

In the CRT algorithm, the modular polynomial $\Phi_{\ell}(X,Y) \mod p_i$ is
computed for many ($O(\ell)$) small primes $p_i$ (with $\log p_i \approx \log
\ell$). The Chinese remainder theorem is then used to reconstruct $\Phi_{\ell}(X,Y) \mod \prod
p_i$.
To obtain a quasi-linear algorithm, $\Phi_{\ell}(X,Y) \mod p_i$  needs to
be computed in $O(\ell^2 \log^{O(1)} p_i)$.
As in the analytic method, this is done by an evaluation-interpolation
approach, but in a careful way to achieve the desired complexity.

Fix a small prime $p=p_i\neq \ell$.
The naive evaluation-interpolation approach would be to select $\ell+1$ elliptic curves $E_i/\F_p$,
and compute all $\ell+1$-isogenous elliptic curves $E_{i,j}/\F_p$.
But given an elliptic curve $E_i/\F_p$, it is already not obvious how to
compute the $\ell+1$-isogenies efficiently: factoring the division
polynomial (of degree $\ell^2$) is too expensive.
A solution is to require that $E_i[\ell]$ has rational torsion; we can then
efficiently sample a basis of it, generate all kernels, and apply the recent square-root Vélu algorithm \cite{bernstein2020faster}, built on V\'elu's formula \cite{velu1971isogenies}, to compute the isogenous curves.
The square-root Vélu algorithm costs $\tilde{O}(\sqrt{\ell})$ arithmetic operations (as opposed to a cost of $O(\ell)$ for the evaluation of the classical V\'elu's formula), so computing
$\Phi_{\ell}(j(E_i), Y)$ in this way already costs $\tilde{O}(\ell \sqrt{\ell})$ arithmetic operations.
Applying the same algorithm to each $E_i$ would result in $\tilde{O}(\ell^2\sqrt{\ell})$  arithmetic operations and is therefore too expensive for the purpose of constructing a quasi-linear algorithm.

Instead, the solution proposed in \cite{broker2012modular} is to use isogeny volcanoes 
to find suitable $\ell$-isogenous curves faster than in $O(\ell)$.
They use the Hilbert class polynomial of a carefully crafted imaginary
quadratic order to ensure that the isogeny volcano has the desired
property.
Their algorithm computes $\Phi_{\ell}$ in $O(\ell^3 \log^{3} \ell \log \log
\ell)$  which is slightly faster than the analytic method, and as mentioned above requires less memory to compute $\Phi_{\ell}$ mod $m$.
However, the complexity proof relies on the Generalised Riemann Hypothesis (this is required to be sure
that there exist small generators of the class group), hence is not unconditional.

\subsection{New algorithmic tools}

The field of isogeny-based cryptography radically changed following the
break of SIDH \cite{castryck2023efficient,eurocrypt-2023-32955,EC23:Robert}. It was soon clear that the tools
(Kani's lemma, Zahrin's trick) used for the break could also be used to
achieve a very efficient representation of isogenies \cite{robert2022evaluating} by
embedding them into higher dimensional isogenies,
which in turn lead to new isogeny-based cryptosystems (for
instance SQISignHD \cite{sqisignhd} and FESTA \cite{festa}).

These new algorithmic tools are not restricted to isogeny-based
cryptography; in \cite{robert2022applications} the second author uses them to outline new
algorithms to compute the endomorphism ring and canonical lifts of ordinary
abelian varieties in polynomial time.
A sketch of a quasi-linear algorithm to compute modular polynomials is
described as well.

We summarize the efficient representation of isogenies of
\cite{robert2022evaluating}, which we will call an {\em HD representation} (for
higher dimensional representation),
as follows:
\begin{theorem}
  \label{thm:rep}
  Let $f: E_1 \to E_2$ be an $\ell$-isogeny between elliptic curves
  over a finite field $\Fq$.
  There exists an HD representation $F$ of $f$ which 
  can be used to  evaluate $f(P)$ for $P \in E_1(\Fq)$
  in time $\Otilde(\log^{\const} \ell)$ arithmetic operations in $\Fq$, for
  some constant $\const$ (not depending on $q$ or $\ell$)\footnote{In
    practice, we can achieve $\const=12$ using \cite[Proposition~2.9]{robert2022applications}.}.

  If $E_1$ has full rational $2^n$-torsion with $2^n > \ell$, there exists an HD-representation for which the evaluation takes $O(\log \ell)$ arithmetic  operations. 
\end{theorem}

  An HD-representation with $\const = 1$, as in the last part of the theorem, is called {\em special HD representation}.
  The key point of the SIDH attacks is that an elliptic curve isogeny $f$ can be embedded into
  a higher dimensional $N$-isogeny as long as we know the evaluation of
  $f$ on a basis of $E_1[N]$, $N>\ell$.
  The first representation in Theorem \ref{thm:rep} follows by taking $N$ powersmooth, and the
  second (which more generally could be used as long as the $2^n$-torsion lies in a small extension of $\Fq$) by taking $N=2^n$.
  Of course, we also have a special HD representation if
  the $3^n>\ell$-torsion is rational and so on.
  We refer to \cite{robert2022evaluating,robert2022applications} for more
  details.

\subsection{Our contributions}

The goal of the current article is to give a fully fledged and rigorous
analysis of the algorithm to compute modular polynomials sketched in \cite[Section 6]{robert2022applications}.
Notably, we prove in \cref{sec:modular} the following theorem, which matches the complexity of~\cite{broker2012modular}.

\begin{theorem} \label{thm:main}
  There exists an unconditional CRT algorithm which computes the modular
  polynomial $\Phi_{\ell}(X,Y)$ in quasi-linear time $O(\ell^3 \log^3 \ell
  \log \log \ell)$.
\end{theorem}

More precisely, we present a rigorous algorithm, along with a faster but
heuristic version, which is the one we implemented.  We remark that this can also be used to compute the reduction modulo~$m$ of $\Phi_{\ell}$ in
quasi-linear space $O(\ell^2 \log (m \ell))$.
To simplify the exposition, we will always assume that $\ell>2$ and the base characteristic~$p>2$, too.

We briefly explain the main ideas behind the algorithms.
The key tool is the following theorem proved in
\cref{sec:deformations-applications}, and
stated here in an informal way:

\begin{theorem} \label{thm:main2}
  Given an HD representation of an $\ell$-isogeny $f: E_1 \to E_2$ over
  a finite field $\Fq=\F_{p^d}$, with $\ell$ prime to the characteristic~$p$,
  and given an $m$-th order deformation of $E_1$ to an elliptic curve
  $\E_1/R$ with $R=\Fq[\epsilon]/(\epsilon^{m+1})$,
  then we can efficiently compute the deformation of $f$ to $\E_1$, that is, $\E_2$ an elliptic curve over $R$ and an isogeny
  $\tilde{f}: \E_1 \to \E_2$ with $\tilde{f} \equiv f \pmod{\epsilon}$
  in $\Otilde(\log^{\const} \ell)$ arithmetic operations in $R$.

  If we have a special HD representation of $f$,
  then we can compute $\tilde{f}$ in $O(\log \ell)$ arithmetic
  operations in $R$.
\end{theorem}

In other words: once we have an efficient representation of an isogeny, we
can efficiently deform it.
Deformation techniques are not new in computer algebra, see for instance
\cite{schost2003computing}. Here we exploit them in order to recover the modular polynomial
$\Phi_{\ell}$ modulo a prime~$p$:

\begin{corollary}
  \label{cor:main}
  Let $\Fq$ be a field with  $\char(\Fq) = p$ and $E_0/\Fq$ be an elliptic curve which has all its $\ell+1$
  $\ell$-isogenies rational.
  Given an HD representation of these isogenies,
  we can compute $\Phi_{\ell} \mod p$ in time
  $\Otilde(\ell^2 \log q)$.
  If we have a special HD representation of all isogenies,
  the cost is $O(\ell^2 \log^2 \ell)$ operations over $\Fq$.
\end{corollary}

\begin{proof}
  The idea is to use the HD representation of the isogenies to deform them
  to $\Fq[\epsilon]/(\epsilon^{m+1})$. When the precision $m$ is high enough, we
  get enough information to reconstruct the full modular polynomial.
  We refer to \cref{sec:modular-Fp} for more details.
\end{proof}

We remark that if $q=p^d$ with $d$ small, then \cref{cor:main} is
quasi-linear in its output size $O(\ell^2 \log p)$.
To apply \cref{cor:main}, we need to find an elliptic curve over a small
extension of $\Fp$ on which we can efficiently compute the $\ell+1$
$\ell$-isogenies.

 The easiest case is when we can find an elliptic curve $E_0/\Fq$ with
    full rational $\ell$-torsion. In this case we compute a basis, compute
    all kernels, and apply Vélu's formula.
    Given a basis of $E_0[\ell]$, computing all the $\ell$-isogenies costs $O(\ell^2)$ arithmetic operations. \footnote{This is not the dominating step of the algorithm, hence there is no need to apply the asymptotically faster square-root V\'elu algorithm.}
    
	For instance, if $p=3 \pmod{4}$ and $\ell \mid p^2-1$, then $E: y^2=x^3+x$ is a supersingular curve over $\Fp$ and
    $E(\F_{p^4})=(\Z/(p^2-1)\Z)^2$, then we can efficiently sample
    a basis of the $\ell$-torsion in $O(\log p)$ operations in $\F_{p^4}$
    (in practice we can work with $E/\F_{p^2}$ or its quadratic twist
    according to whether $\ell \mid p+1$ or $\ell \mid p-1$).
    For technical reasons (to kill the extra automorphisms of $y^2= x^3+x$), we will rather work with a $2$-isogenous curve , that is we take $E_0: y^2=x^3+6x^2+x$. 

    If we add the congruence condition that $2^n \mid p^2-1$ for some $2^n
    > \ell$, so that the $2^n$-torsion of $E_0$ is rational over
    $\F_{p^4}$,
    then we can even use a special HD representation of the
    $\ell$-isogenies.

\begin{proposition}
   \label{prop:pell}
    For $p \in \P^*_\ell = \{p > 11 \text{ prime : } 2^n \cdot \ell \mid p+1, \text{ where } n = \lceil\log_2(\ell)\rceil\}$,
     we can compute $\Phi_{\ell} \mod p$ in
      $O(\log p + \ell^2)$ arithmetic $\Fp$-operations.
\end{proposition}

We prove this proposition in \cref{subsec:generalcase}.
We call a prime $p \in \P^*_\ell$ a suitable CRT prime. By Dirichlet's theorem on primes in arithmetic progressions, we know that there are sufficiently many
such primes, with the appropriate density (see \cref{lem:primes-in-Pl}).
Using these suitable CRT primes, we obtain \cref{thm:main}.

We remark that the idea to use supersingular curves for modular polynomial computation is not new, and in \cite{leroux2023computation} Leroux also gives a heuristic algorithm relying on supersingular isogenies to compute modular polynomials modulo a prime.

In practice, in the implementation, we use a subset $\P_\ell \subset \P^*_\ell$, defined in \cref{subsec:suitable-primes},
where we replace the condition $n=\lceil\log_2(\ell)\rceil$ (which is
enough to ensure that $2^n - \ell$ can be written as a sum of four squares)
by a condition $2^n - c_{\ell} \ell=a^2+4b^2$.
When $p \in \P_\ell$, this stronger condition ensures that we can find a
special HD representation in dimension~$2$ of our isogenies rather
than in dimension~$8$.
We prove this version of \cref{prop:pell} for $p \in \P_\ell$ in
\cref{subsec:modular-poly-a}.
While this does not change the asymptotic complexity of
the algorithm, it improves the constant and greatly simplifies the
implementation.
However, to achieve the quasi-linear CRT algorithm of
\cref{thm:main} when using primes in $\P_\ell$, we need to rely on a heuristic (\cref{heuristic:suitable-n}), which says that we can find a suitable $n$ such that $2^n$ is not too large compared to $\ell$.
In this paper, we will give detailed algorithms for
\cref{thm:main2,cor:main,thm:main,prop:pell} for the special case when $p
\in \P_{\ell}$, because it is the one we implemented, and briefly explain
how to generalize to the general cases.

\subsection{Generalisations and perspectives}
The techniques used in our algorithms, can also be applied to other settings. Here, we highlight some  generalizations.  \smallskip 

\noindent\textbf{Computing $\Phi_{\ell}$ modulo an arbitary prime~$p$.}
 In this paper, we explain the computation for primes $p \in \P_\ell^*$. This may be extended to general primes~$p$ as follows. We also use supersingular curves, because
    all their $\ell$-isogenies are rational over $\F_{p^2}$.
    We first need to sample a supersingular elliptic $E_0/\F_{p^2}$
    with explicit known endomorphism ring and endomorphism action. 
    This can be seen as a precomputation
    depending only on $p$ and not on $\ell$.
    For instance, we can apply Br\"{o}ker's algorithm \cite{broker2009constructing},
    which is in $O(\log^3 p)$ under GRH.

    The explicit endomorphism ring action of $E_0$ allows us to compute the
    $\ell+1$ $\ell$-isogenies in polynomial time in $\log \ell+\log p$ by
    \cite{wesolowski2022supersingular} under GRH, or unconditionally using the more recent CLAPOTIS method
    \cite[Remark~2.10]{DRclapotis}.
    Namely, we compute the ideals in $\End(E_0)$ corresponding to the
    $\ell+1$ $\ell$-isogenies, and then we 
    use CLAPOTIS to compute an HD representation of the $\ell$-isogenies
    associated to these ideals in $\Otilde((\log \ell+\log p)^{\const})$
    arithmetic operations by ideal.

    This initialisation is only interesting when $\log p$ is sufficiently
    small compared to $\ell$; it is thus not suitable in applications like point counting where we
    usually have $\log p \approx \ell$.
    (Of course conversely if $\log p$ is sufficiently large compared to $\ell$, it is faster to
    just compute $\Phi_{\ell}$ directly.)
    An alternative initialisation approach is described in
    \cite[6.2]{robert2022applications} with better complexity in $\log p$,
    and another heuristic approach is described in \cite{leroux2023computation}.

From the discussion above, we get:
\begin{corollary}
  Let $p$ be a fixed prime. 
  Assume that we are given an explicit supersingular curve $E_0/\F_{p^2}$
  with known full endomorphism ring (as mentionned above such an $E_0$ can
  be computed in $\Otilde(\log^3 p)$ under GRH).
  Then there is a quasi-linear algorithm to
  compute $\Phi_{\ell} \mod p$ in $O(\ell^2 (\log p + \log \ell)^{\const})$,
  for some constant $\const$ which does not depend on $\ell$ or $p$.

  In particular, under GRH, we obtain an algorithm that can compute
  $\Phi_{\ell} \mod p$ in $\Otilde_p(\ell^2)$,
  where the notation $\Otilde_p$ means that the constants depend on~$p$.
\end{corollary}

  \smallskip
\noindent \textbf{$p$-adic lifting.}
  One can extend \cref{thm:main2} to more general Artinian rings with residue
  field~$\Fq$, notably rings of the form
  $R=\Zq[\epsilon]/(\epsilon^{m_1}, p^{m_2})$.
  This allows us to combine our horizontal deformation with a $p$-adic
  lifting, and give 
a $p$-adic lifting version of \cref{cor:main}
which results in
  an alternative way to compute the modular polynomial
  $\Phi_{\ell}$ in quasi-linear time.
  However, the CRT algorithm is better because it can be used to compute
  $\Phi_{\ell} \mod m$ by the explicit CRT.

  \smallskip
\noindent \textbf{Modular polynomials in higher dimension.}
In this paper, for simplicity we confine ourselves to computing
deformations
of isogenies of elliptic curves. Using theta models for abelian varieties,
our methods to compute these deformations efficiently (\cref{thm:main2}) extend naturally to higher dimension.
In subsequent work, we will explain how this paves the way to computing
higher dimensional Siegel modular polynomials in quasi-linear time. Such a
complexity was conjectured in \cite[Conjecture~5.3.14]{hdr}.
The reason we focus on the dimension~$1$ case in this current paper is by
lack of space, for
the simplicity of the exposition\footnote{The main difficulty in higher
  dimension is that we now need to reconstruct multivariate rational
  functions from power series rather than simply polynomials; and also take
into account the algebraic equations between the modular invariants.}, and also because we have only implemented
this case so far.

For modular polynomials in higher dimension, there is currently no known quasi-linear analytic method.
An important step in this direction is achieved in a recent work by Elkies and Kieffer, who developed a quasi-linear algorithm to evaluate theta functions in quasi-linear time in the precision in all dimensions. However,
to compute the modular polynomial efficiently, one would also need a
quasi-linear algorithm for computing period matrices from theta constants,
because the evaluation-interpolation approach requires to select
suitable points of interpolation to obtain a fast multivariate interpolation algorithm. 
Such an algorithm is only known in dimensions~$1$ and~$2$.
The dimension~$2$ case follows from heuristic work of Dupont
\cite{dupont2006moyenne}, and was then rigorously proven by Kieffer in
\cite{kieffer2020sign}.
Kieffer thus describes in \cite{kieffer2020evaluating} a rigorous algorithm to evaluate modular
polynomials in dimension~$2$, which combined with \cite{dupont2006moyenne, milio2015quasi}
gives a rigorous analytic algorithm in dimension~$2$ to compute the full
modular polynomial.

Likewise, the CRT approach of \cite{broker2012modular} crucially uses the
volcano structure;
but isogeny graphs of
principally polarised abelian varieties over a finite field are more
complex, so it is not obvious how to extend this approach to higher dimensions.

By contrast, as we have mentioned, our deformation method generalises in a
natural way to higher dimensions.
We expect that similar deformation techniques could also be used to compute more
general modular correspondences or actions of Hecke operators on modular forms.

\subsection{Outline}
Section \ref{sec:deformations-theory} contains theoretical background on deformations. The results of that section are made explicit in Section \ref{sec:deformations-applications}, where we describe different methods for computing with deformations.  
In Section \ref{sec:modular-Fp}, we present an explicit algorithm to compute modular polynomials over finite fields.  Building on this, we describe a CRT algorithm for computing the modular polynomial in Section \ref{sec:modular}. 

%% file: deformations.tex
\section{Deformations and isogenies} \label{sec:deformations-theory}

In this section, we provide an overview on the topic of deformations of principally polarised abelian varieties, and in particular we discuss deformations of isogenies. For more detailed explanations, the reader is referred to \cite{hartshorne2010deformation, sernesi2007deformations}.
Throughout, $k$ denotes a field with characteristic $\char(k) = p > 2$, and $R$ is an Artinian ring of the form $R = k[\epsilon]/(\epsilon^{m+1})$ for some integer $m\geq 0$.\footnote{The exponent $m+1$ is chosen, since we work with $m$-th order deformations (Definition \ref{def:deformation}).}

\subsection{Deformations of (principally polarised) abelian varieties}

Recall that an abelian variety is a group scheme over a field $k$ which is also a proper, geometrically integral variety over $k$. This fits into the following more general framework of abelian schemes.

\begin{definition} \label{def:abelian-scheme}
	An {\em abelian scheme} over a scheme $S$ is a group scheme $\A \to S$ which is smooth, proper and has geometrically connected fibres. If $S = \Spec(k)$, then we call $\A$ an abelian variety.
\end{definition}

In our applications, we always have $S = \Spec(R)$ with $R = k[\epsilon]/(\epsilon^{m+1})$ as above, but everything extends to a general local Artinian ring $R$ with maximal ideal $m_R$ and residue field $R/m_R=k$. Note that in this setting, along with the map $\Spec(k) \to \Spec(R)$ induced
by the projection $R \to R/(\epsilon)$, there is also a map $\Spec(R) \to \Spec(k)$ induced by the canonical inclusion $k \hookrightarrow R$.

\begin{example}
	Let $\E \to \Spec(R)$ be an abelian scheme of dimension $1$, in other words an elliptic curve over $R$.  As is the case for elliptic curves over fields, we can make Definition \ref{def:abelian-scheme} more explicit in this situation.  Since $R$ is local, it can be represented as a subscheme of $\PP_R^2$ defined by an equation of the form
	\[
	\E: Y^2Z + a_1 XYZ + a_3 YZ^2 = X^3 + a_2 X^2Z + a_4 XZ^2 + a_6 Z^3 \,  
	\]
	with $a_1,a_2,a_3,a_4,a_6 \in R$. 
	Note that $\Spec(R) = \{(\epsilon)\}$. Therefore, the scheme $\E$ has only one fibre $E = \E_{(\epsilon)}$. It is obtained by reduction modulo $(\epsilon)$.
	
	Note that points of $\E$ are sections $s: \Spec(R) \to \E$. More explicitly, we may also view a point $P \in \E(R)$ as usual by three projective coordinates, i.e. ${P = (x:y:z) \in \E \subset \PP_R^2}$.
\end{example}

\begin{definition} \label{def:deformation}
	Let $A$ be an abelian variety over $k$. We say that an abelian scheme $\A \to \Spec(R)$ is an {\em $m$-th order deformation} of $A$ if its special fibre $\A \times_R k \to \Spec(k)$ is isomorphic to $A$.
\end{definition}

In our setting, along the special fibre map $\Spec(k) \to \Spec(R)$ induced by $R \to R/(\epsilon)$, there is also a map $\Spec(R) \to \Spec(k)$ induced by the canonical inclusion of $k$ into $R$.
For any abelian variety $A$, there is a trivial deformation given by $\A = A \times_k \Spec(R)$. In general, the group structure of a deformation $\A$ of $A$ is not obvious. However, we can describe the \'etale part of the torsion group explicitly which is explained in the next remark.

\begin{remark} \label{rem:group-strucutre}
	Let $N$ be an integer coprime to $p$, and consider a deformation $\A \to \Spec(R)$ of an abelian variety $A$. 
	Since $\A[N]/\Spec(R)$ is étale and $R$ is Henselian, we have a canonical isomorphism
	$\A[N](R) \cong A[N](k)$ (see~\stackcite{04GG}).
	This isomorphism will be made explicit in Algorithm \ref{algo:lift_point} for elliptic curves. 
	In fact, because $R$ is Henselian, the functor which associates to a
	finite étale cover $\mathcal{X}/\Spec(R)$ its special fibre $X/\Spec(k)$ is an equivalence
	of categories between finite \'etale covers of $\Spec(R)$ and finite étale covers of $\Spec(k)$ (see~\stackcite{0A48}).
	Thus $A[N]$ deforms uniquely to $\A[N]/\Spec(R)$, in our setting this 
	is simply given by the trivial deformation.
	
	We remark that in our applications, we  always work with points of $\A[N]$, and only compute separable isogenies. In these cases, we can use the standard addition laws and formulae for isogeny computation known for abelian varieties over fields.
	Indeed, in \cite[\S~6]{MumfordOEDAV2}, Mumford construct the universal
	abelian scheme of level~$n$ over $\Z[1/n]$ via Riemann's relations,
	in level divisible by $8$ (the construction was then extended by Kempf
	to level divisible by $4$).
	In particular, the addition law and $\ell$-isogeny formulas which use
	these Riemann relations have good reduction over
	$\Z[\frac{1}{2 \ell}]$, hence are valid over $R$ as long as $p \ne
	2, \ell$.

        Although we won't need it, we can make explicit the full structure
        of $\A(R)$ as follows. The kernel of the reduction map $\A(R) \to
        A(k)$ is given by $\A(p)^0(R)$, where $\A(p)^0$ is the connected
        component of the $p$-divisible group of $\A$.
        We have $\A(p)^0(R)=\tilde{\Gamma}(m_R)$ where $\tilde{\Gamma}$ is
        the formal Lie group associated to $\A$.
        Furthermore, since $k$ is perfect, the connected-étale sequence
        splits over $k$, hence we have an exact sequence
        $0 \to \A(p)^0(R)=\tilde{\Gamma}(m_R) \to \A(p)(R) \to
        \A(p)_{\text{\'etale}}(R)=A(p)_{\text{\'etale}}(k) \to 0$.
        We refer to \cite{tate1967p} for more details.
\end{remark}

Note that, as proved by Grothendieck (see~\cite[\S 2.2]{oort1971finite}), the deformation space of a $g$-dimensional abelian variety has dimension $g^2$ which is equal to the dimension of the tangent space.  
We now specialize to the case of principally polarised abelian varieties. 
Recall that a {\em polarisation} $L$ on an abelian variety $A$ is given by an ample line bundle (defined up to translation and potentially over a separable field extension of $k$) which defines a rational morphism to the dual of $A$, $\phi_L: A \to \hat{A}$. And it is called {\em principal} if $\phi_L$ is an isomorphism. 

\begin{definition}
  Let $(A,L)$ be a principally polarised abelian variety over $k$. We say that $(\A, \L)$ is an {\em $m$-th order deformation} of $(A,L)$, if $\A$ is a deformation of $A$ and $\L$ is an ample line bundle on $\A$  with $\L \times_R \Spec(k) \cong L$.\footnote{A slight technicality: deforming the polarisation means deforming the isogeny $\phi_L$. However, if $\phi_L$ is induced by a line bundle over $k$ (rather than over its separable closure), and $\phi_L$ deforms to $R$, then $L$ does too. This is because the obstruction to descending $L$ is given by a smooth torsor, which has a point over $k$ by assumption, so over $R$ too since $R$ is Henselian.}
\end{definition} 

If the polarisation is clear from the context, we write $A$ (resp. $\A$) instead of $(A,L)$ (resp. $(\A,\L)$).
The deformation space of principally polarised abelian varieties has dimension $g(g+1)/{2}$. More precisely, Grothendieck and Mumford proved that for a principally polarised abelian variety $A/k$, the functor of deformations of $A$ is pro-representable by $k[[t_1, \dots, t_{g(g+1)/2}]]$ (see~\cite[\S 2.3]{oort1971finite}). 
Concretely, it is given by the completion at $A$ of the moduli space $\Am_g/k$ of principally polarised abelian varieties\footnote{To handle deformations to a general Artinian ring, we would need to use the completion $\hat{\Am}_{g,A}/\Z \iso \Z_q[[t_1, \dots, t_{g(g+1)/2}]]$ at $A$ of the moduli space over $\Z$ rather than over $k$.}.
Essentially, this means that any $m$-th order deformation $\A$ of $A$ corresponds to a unique ring homomorphism \begin{equation} \label{eq:deformation-map}
\psi_{\A}: k[[t_1, \dots, t_{g(g+1)/2}]] \to R.
\end{equation}
 This motivates the following definition.

\begin{definition}
	Let $A$ be a principally polarised abelian variety over $k$, and $\A$ an $m$-th order deformation of $A$. Then we say that $$\lambda_1 = \psi_{\A}(t_1), \dots, \lambda_{g(g+1)/2} = \psi_{\A}(t_{g(g+1)/2})$$ with $\psi_A$ as in Eq. \ref{eq:deformation-map} are {\em deformation parameters} of $\A$.
\end{definition}

Note that the definition of the deformation parameters depends on the local representation $k[[t_1, \dots, t_{g(g+1)/2}]] $ of the moduli space at $A$. Below, we describe a canonical choice for the case $g=1$.

\begin{example} \label{exa:deformation-parameters}
	In the case of elliptic curves, the deformation space is one-dimensional, hence a deformation is defined by a single deformation parameter $\lambda$.
	
	 Let $E$ be an elliptic curve and suppose that $j(E) \neq 0,1728$. For an $m$-th order deformation $\E$ of $E$, we can define the deformation parameter \[
	 \lambda_{\E} = j(\E) - j(E) \in (\epsilon) \triangleleft k[\epsilon]/(\epsilon^{m+1})
	 \]
	 associated to the $j$-invariant.
	 
	 Note that as long as $j(E) \neq 0,1728$, we can construct the universal elliptic curve over the universal deformation ring. More explicitly, for a given deformation parameter $\lambda \in (\epsilon)$, we can consider 
	  \[
	  \E: y^2 = x^3+ax+b, \quad \text{where } a=b=\frac{27 (j(E)+\lambda)}{4(1728-(j(E)+\lambda))},
	  \]
	  which defines an elliptic curve over $R$ with $j$-invariant $j(\E) = j(E) + \lambda$.
	  Over $j(E)=0, 1728$ the universal elliptic curve only exists as
	  an algebraic stack \cite{deligne_SchemasModulesCourbes1973}, which is why we are going to use $y^2 = x^3+6x^2+x$ rather than $y^2 = x^3+x$ as the initial point in our algorithms. 
\end{example}

\subsection{Deformations of isogenies}

Let $f:A \to A'$ be an isogeny of principally polarized abelian varieties over a field $k$. Here, we discuss deformations of such an isogeny to $R$. In particular, we are interested in deformations of product isogenies, i.e. isogenies between decomposable principally polarized abelian varieties, and the behaviour of deformation parameters under isogenies.

Recall that a morphism of group schemes  $f:\A \to \A'$ over $\Spec(R)$ is an isogeny if it is surjective and its kernel is a flat finite $\Spec(R)$-group scheme.

\begin{proposition} \label{prop:lift-isogenies}
	Let $f: A \to A'$ be a separable isogeny of principally polarised abelian varieties defined over $k$. Consider an $m$-th order deformation $\A$ of $A$. Then up to isomorphism there exists a unique deformation $\A'$ of $A'$ so that $f$ lifts to an isogeny $\tilde{f}: \A \to \A'$ over $R$.
\end{proposition} 

This is a classical result, which stems from the fact that the forgetful
map $\Am_g(\Gamma_0(\ell)) \to \Am_g$ is étale. Recall that $\Am_g(\Gamma_0(\ell))$ denotes the moduli space of principally polarized abelian varieties with $\Gamma_0(\ell)$-structure, i.e. the elements are principally polarized abelian varieties together with the kernel of an $\ell$-isogeny.
Since we need to lift isogenies explicitly in our applications, we provide a short constructive proof below.

\begin{proof}[Proof of Proposition \ref{prop:lift-isogenies}]
  Let $G$ be the kernel of the isogeny $f:A \to A'$. By assumption, $f$ is separable, hence $G \subset A[N]$ where $N$ is coprime to $p$. Now consider the deformation $\A$ of $A$ to $R$. Since
  $R$ is Henselian, it has the same finite étale covers as $k$.  It follows that there is a unique lift  $\tilde{G} \subset \A[N]$ of $G$.
  We can set $\A' = \A/\tilde{G}$ which is again a group scheme. Necessarily $\A'$ is a deformation of $A'$, and the morphism $\tilde{f}: \A \to \A/\tilde{G}$ is an isogeny lying above $f$. 
\end{proof}

The following proposition tells us how deformations behave under isogenies. To obtain a description in terms of deformation parameters, we assume that some canonical choice on the local representation of the moduli space has been made. We simply refer to {\em the} deformation parameters, when we mean the deformation parameters induced by this choice. 

\begin{proposition} \label{prop:deformation-parameters}
	Let $f:A\to A'$ be a separable isogeny of principally polarized abelian varieties over $k$ and $m\geq 1$, then there exist polynomials \[
	h_1, \dots, h_{g(g+1)/2} \in k[x_1,\dots, x_{g(g+1)/2}]
	\] of degree at most $m$ with the following property:\\	
	For any $m$-th order deformation $\A$ of $A$ with parameters $(\lambda_1, \dots, \lambda_{g(g+1)/2})$, the deformation $\A'$ with parameters $(\lambda'_1, \dots, \lambda'_{g(g+1)/2})$, where 
	\[
	\lambda'_i = h_i(\lambda_1,\dots, \lambda_{g(g+1)/2})
	\]is the unique deformation of $A'$ for which $f$ lifts to an isogeny $\tilde{f}:\A \to \A'$. 
\end{proposition}

\begin{proof}
	Let $\hat{\Am}_{g,A} = k[[t_1,\dots,t_{g(g+1)/2}]]$ and $\hat{\Am}_{g,A'} = k[[t_1',\dots,t_{g(g+1)/2}']]$ be the local completions of the moduli space $\Am_g$ at $A$ and $A'$, that induce our canonical choice of deformation parameters. The isogeny $f$ 
	corresponds to a point $(A, \Ker f) \in \Am_g(\Gamma_0(\ell))$.
	The modular correspondence $\corresp: \Am_g(\Gamma_0(\ell)) \to
	\Am_g \times \Am_g$ is given by two étale projection maps.
	This induces an (analytic, i.e. continuous) isomorphism of the completion of
	$\Am_g(\Gamma_0(\ell))$ at $(A, \Ker f)$ with
	$\hat{\Am}_{g,A}$ and $\hat{\Am}_{g,A'}$ respectively.
	In particular, $f$
	induces an isomorphism 
	\(
	\phi_f: \hat{\Am}_{g,A} \to \hat{\Am}_{g,A'}
	\). 
	Thus, we can write 
	\[
	t_i' = \tilde{h}_i \in k[[\phi_f(t_1),\dots,\phi_f(t_{g(g+1)/2})]]
	\] 
	for each $i \in \{1,\dots, g(g+1)/2\}$. Now the polynomials $h_1, \dots, h_{g(g+1)/2}$ from the statement of the proposition are obtained by truncating these formal power series at degree $m$. 
\end{proof}

\begin{remark}
	The modular polynomial $\Phi_{\ell}$ describes the image of the modular
	correspondence $\corresp: \Am_1(\Gamma_0(\ell))/k \to \Am_1/k \times \Am_1/k$,
	it can also be seen as $\corresp$ evaluated on the generic point of
	$\Am_1/k$.
	From this point of view \cref{cor:main} is natural: we start with the
	modular correspondence $\corresp$ evaluated at one point $E \in \Am_1/k$, i.e. with
	all $\ell$-isogenies starting from $E$.
	More precisely, the fibre
	$(\pi_1 \circ \corresp)^{-1}(E)$ gives the isogenies and
	evaluating
	$(\pi_2 \circ \corresp)$ on the fibre gives the codomains.
	We then evaluate $\corresp$ at the universal deformation of $E$ at a high
	enough precision to recover $\corresp$ on the generic point.
\end{remark}

\subsection{Deformations of product isogenies} 

We now restrict ourselves to the special case of product isogenies in dimension $2$, that is isogenies between decomposable principally polarized abelian surfaces. Let us first recall Kani's Lemma which explains the relation between isogeny diamonds and product isogenies. 

\begin{definition} \label{def:isogeny-diamond}
	Let $E, E_a, E_b, E_{ab}$ be elliptic curves and $f,g,f',g'$ isogenies of degree $d_a = \deg(f) = \deg(f')$ and $d_b = \deg(g) = \deg(g')$ that fit into the following commutative diagram.
	\begin{equation} \label{eq:isogeny-diamond}
		\begin{tikzcd}
			E \arrow[r, "f"] \arrow[d, "g"] & E_a \arrow[d, "g'"] \\
			E_b \arrow[r, "f'"]& E_{ab}
		\end{tikzcd}
	\end{equation}
	
	Then $g \circ f' = f \circ g'$ is called a {\em $(d_a,d_b)$-isogeny diamond}. 
\end{definition}

\begin{lemma}[Kani's Lemma \cite{kani1997number}] \label{lem:kani}
	Let $g \circ f' = f \circ g'$ be a $(d_a,d_b)$-isogeny diamond and denote $N = d_a + d_b$. Then $F = \begin{pmatrix}
		f & \hat{g'}\\ -g & \hat{f'}
	\end{pmatrix}$ is an $N$-isogeny 
	$F : E \times E_{ab} \to E_a  \times E_b$.  If $\gcd(d_a,d_b) = 1$, then
	\begin{align*}
		\ker(F) 
		=&~ \left\langle \left(-\hat{g}(P_N), f'(P_N)\right), \left(-\hat{g}(Q_N), f'(Q_N)\right)  \right\rangle, 
	\end{align*}
	where $E_a[N] = \langle P_N, Q_N\rangle$.
\end{lemma}

Note that Kani's Lemma can be generalised to abelian varieties of arbitrary dimension, see \cite{EC23:Robert}, and is the key to the HD representation of isogenies in \cite{robert2022evaluating}.

The following corollary is an easy consequence of the unique lifting of isogenies (Proposition \ref{prop:lift-isogenies}) and Kani's lemma. It describes when product isogenies lift to product isogenies.

\begin{corollary} \label{cor:unique-lift-isogeny-diamond}
	Let $E,E_{ab}$ be elliptic curves over $k$ and  $F: E \times E_{ab} \to E_a \times E_b$ a product isogeny. Consider an $m$-th order deformation $\E$ of $E$. Then up to isomorphism there exist unique deformations $\E_{ab}$, $\E_a$, $\E_b$ of $E_{ab},E_a,E_b$, so that $F$ lifts to a product isogeny
	\[
	\tilde{F}: \E \times \E_{ab} \to \E_a \times \E_b.
	\]
	
\end{corollary}

\begin{remark}
	Let $F: E \times E_{ab} \to E_a\times E_b$ as in Corollary \ref{cor:unique-lift-isogeny-diamond}. Further let  $\E$ and $\E_{ab}$ be arbitrary deformations of $E$ and $E_{ab}$. Then the isogeny $F$ lifts uniquely to an isogeny $\tilde{F}: \E \times \E_{ab} \to \A$, where $\A$ is a deformation of $E_a \times E_b$. To understand, when $\tilde{F}$ is a product isogeny, we recall from Kani's lemma that any such isogeny is induced by an isogeny diamond as in Eq. \ref{eq:isogeny-diamond}. This means that the deformation $\A$ is a product if and only if there is an isogeny $\tilde{h}: \E \to \E_{ab}$ lifting the isogeny $f' \circ g = g' \circ f$ underlying the product isogeny $F$. 
\end{remark}

The next result describes a convenient tool for computing deformations of product isogenies. This makes use of the Siegel modular cusp form $\chi_{10}$. Recall that 
\[
\chi_{10}(\Jac(C)) =  -2^{-12} \cdot \disc(f), \qquad \chi_{10}(E\times E') = 0,
\]
where $C:y^2=f(x)$ is a genus-$2$ curve, and $E$, $E'$ are elliptic curves.
In particular $\chi_{10}$ can be used to distinguish decomposable from indecomposable elements in the moduli space.

The following corollary is a consequence of Proposition \ref{prop:deformation-parameters} applied to product isogenies. To make a precise statement, we require that isogenies are {\em normalised}. Essentially, this means that the isogeny acts as the identity on the basis of differentials.

\begin{corollary} \label{cor:chi10}
	Let $E,E_{ab}$ be elliptic curves over a field $k$ and  $F: E \times E_{ab} \to E_a \times E_b$ a product isogeny. Then there exists a polynomial $h\in k[x_1,x_2]$ of degree at most $m$ with the following property:
	
	Let $\E$ and $\E_{ab}$ be arbitrary $m$-th order deformations of $E$ and $E_{ab}$ with deformation parameters $\lambda = j(\E)-j(E)$ and $\lambda_{ab} = j(\E_{ab})  - j(E_{ab})$, respectively. Further let $\tilde{F}:\E \times \E_{ab} \to \A$ be a lift of $F$ and assume the isogeny is {\em normalised}. Then 
	\[
	\chi_{10}(\A) = h(\lambda, \lambda_{ab}).
	\] 
\end{corollary}

\begin{proof}
	On a decomposable abelian surface $A = E \times E_{ab}$, a canonical choice of deformation parameters is given by choosing two deformation parameters $\lambda$ and $\lambda_{ab}$ for $E$ and $E_{ab}$, respectively. Note that altering only these two deformation parameters preserves the product structure on the resulting deformation. 
	Let $S \subset \hat{\Am}_{2, E \times E_{ab}}$ be the subring of the
	deformation ring given by product deformations. 

  A difference with \cref{prop:deformation-parameters} is that $\chi_{10}$
  is a modular form rather than a modular function.
        Let $\chi_{10,\ell}$ be the modular form which associates
        to a normalised $\ell$-isogeny $F: (A,\omega_A) \to (B,\omega_B)$
        the value $\chi_{10}(B,\omega_B)$.
        (Normalised means that ${F}^\ast \omega_B=\omega_A$.)
        This is a modular form of level $\Gamma_0(\ell)$, hence a section
        of a power of the Hodge line bundle $\hodge^{10}$ on
        $\Am_2(\Gamma_0(\ell))$. \footnote{Note that over $\CC$, $\chi_{10,\ell}$ corresponds to the modular forms of level $\Gamma_0(\ell)$: $\tau \mapsto \chi_{10}(\tau/\ell)$.}

  Now given curve equations for $\E \times \E_{ab}/R$, we get a canonical
  basis of differentials $((dx/y, 0), (0, dx_{ab}/y_{ab}))$, hence a 
  trivialisation of the Hodge bundle $\hodge$ above $S$, i.e. an
  isomorphism $S \to S \tens_{\Am_2} \hodge$.
  We can use the modular correspondence $\corresp$ as in
  \cref{prop:deformation-parameters}
  to see
  $S$ as a deformation ring of the isogeny $F$,
  and pulling  back the trivialisation of the Hodge line bundle
  defined above,
  we can interpret $\chi_{10,\ell} \in \Gamma(S \tens_{\Am_2} \hodge^{10})$
  as an element of $S$, given by a polynomial $h$.
\end{proof}

\begin{remark}
  In \cref{sec:deformations-applications}, we will apply Corollary \ref{cor:chi10} to a $(2^n,2^n)$-isogeny $F$ induced by an isogeny diamond. The isogeny computation consists of a gluing
  isogeny, followed by a chain of Richelot isogenies, followed by a
  splitting isogeny, hence we need to ensure that each step is normalised.
  The differentials are implicitly kept track off by the curve equation:
  to an elliptic curve $y^2=x^3+ax+b$ we associate the differential $dx/y$,
  and to an hyperelliptic curve of genus~$2$ the differentials $(dx/y, x
  dx/y)$. Richelot formulas work at the level of curve equations and give
  normalised formulas.
\end{remark}

%% file: explicit-deformations.tex
\section{Computing with deformations} \label{sec:deformations-applications}
In this section we explain different algorithms related to the computation of deformations. This includes the lifting of torsion points to deformed elliptic curves (Algorithm \ref{algo:lift_point}), as well as the lifting of $(2,2)$-isogeny chains (Algorithm \ref{algo:Richelot-chain}). Further, we present a method for lifting product isogenies (Algorithm \ref{algo:lift_isogeny_diamond}) which lies at the heart of our algorithms for the computation of modular polynomials. The section concludes with the deformation of general isogenies and proves  \cref{thm:main2}.

As in the previous section, $k$ is a field with characteristic $p \neq 0$, and $R = k[\epsilon]/(\epsilon^{m+1})$.

\subsection{Arithmetic in $R = \FF_q[\epsilon]/(\epsilon^{m+1})$}
\label{subsec:arithmetic-in-R}
In practice, we work with a finite field $k = \FF_q$ in all algorithms. Note that computing in the Artin ring $R$ is equivalent to computing in the formal power series ring $\FF_q[[\epsilon]]$ with precision $m+1$.  

Throughout, we use the Sch\"onhage-Strassen bound
\[
\mathsf{M}(n) = O(n\log n \log\log n)
\]
to describe the complexity for the multiplication of two $n$-bit integers.\footnote{We remark that there exists an algorithm for integer multiplication in $O(n\log n)$ by Harvey and van der Hoeven, \cite{harvey2021integer}. Here, we however use the bound $\mathsf{M}(n)$ defined above, in order to make it easier to compare our results to those in \cite{broker2012modular}.} Further, we set
\begin{gather*}
  \mathsf{M}(\FF_p, n)=O(\mathsf{M}(n(\log p+\log n))), \quad \mathsf{M}(\FF_q, n) = O(\mathsf{M}(\FF_p, en)),\\
  \mathsf{M}(R) = O(\mathsf{M}(\FF_q, m))
\end{gather*}
where $q=p^e$, and $\mathsf{M}(R,d)$ denotes the complexity of the multiplication of
two polynomials of degree~$d$ over the ring~$R$.
We also let $\mathsf{M}(R)=\mathsf{M}(R,1)$ the complexity of the multiplication of two
elements in $R$.
These bounds can be obtained from Kronecker substitution \cite{kronecker1882grundzuge}. 
Similarly, the multiplication of two polynomials of degree $O(d)$ over $R$ costs
\[
  \mathsf{M}(R,d) = O(\mathsf{M}(\FF_q, dm))=O(\mathsf{M}(m \,e\,d(\log p + \log(m \,e\,d)) )).
\] 
Moreover, we already note that later we will have $m \approx d \approx \ell$, $e \in \{1,2\}$ and $\log(p) \in O(\log(\ell))$ for some integer $\ell$, hence 
\begin{gather*}
  \mathsf{M}(R) = O(\mathsf{M}(\ell \log \ell))=O(\ell \log^2\ell \log\log \ell), \\ \mathsf{M}(R,\ell) = O(\mathsf{M}(\ell^2 \log \ell))=O(\ell^2 \log^2\ell \log\log \ell).
\end{gather*}

Inversions in $R$ can be computed with the same asymptotic complexity, for instance by using {\em Newton lifts}. Remarkably, this approach also allows us to lift roots of polynomials over $k$ to roots of polynomials over $R$. Since Newton lifts play an important role in our algorithms, this standard method is presented in Algorithm \ref{algo:Newton}. 

\begin{algorithm}[h!]
	\caption{\texttt{Newton\_lift}}\label{algo:Newton}
	\begin{flushleft}
	\textbf{Input:} An element $\alpha \in k$ and a polynomial $f \in R[x]$ such that $f(\alpha) \equiv 0$ and $f'(\alpha) \not\equiv0 \pmod{(\epsilon)}$.\\
	\textbf{Output:} An element $\tilde{\alpha} \in R$ with $\tilde{\alpha} \equiv \alpha \pmod{(\epsilon)}$ and ${f}(\tilde{\alpha}) = 0$. 
	\end{flushleft}
	\begin{algorithmic}[1]
		\For{$r \gets 1, \dots, \lceil\log_2(m+1)\rceil$}
		\State ${\alpha} \gets {\alpha} + O(\epsilon^{\min(2^r,m+1)})$
		\State ${\alpha} \gets {\alpha} - \frac{{f}({\alpha})}{{f}'({\alpha})}$
		\EndFor
		\State \Return ${\alpha}$
	\end{algorithmic}
\end{algorithm}

\subsection{Lifting torsion points}

Let $A$ be an abelian variety over $k$. Given a point $P \in A(k)$, and an $m$-th order deformation $\A$ of $A$, there exist multiple points $\tilde{P} \in \A(R)$ reducing to $P$. On the other hand, there is an isomorphism $\A[N] \cong A[N]$ if $p \nmid N$. In order to be able to lift isogenies, we need to make this isomorphism explicit. The case of elliptic curves is covered by Algorithm \ref{algo:lift_point} and the special case of abelian surfaces and $N=2$ is covered by Algorithm \ref{algo:lift_2torsion}. 

\begin{lemma}
	On input a point $P \in E[N]$ with $N$ coprime to $\char(k)$, $E$ an elliptic curve over $k$ and an $m$-th order deformation $\E$, Algorithm \ref{algo:lift_point} returns the unique lift $\tilde{P} \in \E[N]$ of $P$. For fixed $N$, the algorithm runs in time $O(1)$ over $R = k[\epsilon]/(\epsilon^{m+1})$.
\end{lemma}

\begin{proof}
	Since $\char{k} \nmid N$, there is a unique lift $\tilde{P} =(\tilde{x},\tilde{y})$ of $P=(x_0,y_0) \in E[N]$ which is also an $N$-torsion point. 
	Let $f_{\E,N}$ denote the $N$-th division polynomial of $\E$. Then $f_{\E,N}(x_0)\equiv 0 \pmod{(\epsilon)}$, and we can lift $x_0$ to a root $\tilde{x}_0$ of $f_{\E,N}$ using the method \texttt{Newton\_lift} (Algorithm \ref{algo:Newton}). 
	The corresponding $y$-coordinate $\tilde{y}_0$ lying above $y_0\in k$ is computed by another call to  \texttt{Newton\_lift}.
	
	It is clear that $\tilde{P} = (\tilde{x}_0,\tilde{y}_0)$ is in $\E[N]$. And the running time of the algorithm is determined by the running time of \texttt{Newton\_lift}.
\end{proof}

\begin{algorithm}[h!]
	\caption{\texttt{lift\_point}}\label{algo:lift_point}
	\begin{flushleft}
	\textbf{Input:} Elliptic curve $E$ over $k$, a point $P \in E(k)[N]$ with $p \nmid N$, and an $m$-th order deformation $\E : y^2 =f(x)$. \\
	\textbf{Output:} The lift $\tilde{P} \in \E[N]$ of $P$. 
	\end{flushleft}
	\begin{algorithmic}[1]
		\State $(x_0,y_0) \gets P$
		\State $f_{\E,N}$ the $N$-th division polynomial of $\E$.
		\State $\tilde{x}_0 \gets \texttt{Newton\_lift}(x_0, f_{\E,N})$
		\State $\tilde{y}_0 \gets \texttt{Newton\_lift}(y_0, y^2-f(\tilde{x}_0))$
		\State \Return $(\tilde{x}_0,\tilde{y}_0) \in \E[N]$
	\end{algorithmic}
\end{algorithm}

There are different ways to generalise Algorithm \ref{algo:lift_point} to higher dimensions. In our applications, we only need to lift $2$-torsion points of abelian surfaces. On decomposable abelian surfaces, we may just apply Algorithm \ref{algo:lift_point} on every component. Algorithm \ref{algo:lift_2torsion} describes a general method for lifting $2$-torsion points on abelian surfaces. For simplicity, we represent $2$-torsion points as pairs $P = \{\alpha_1,\alpha_2\} \in A(k)[2]$. If $A = \Jac(C)$, this means $P = [(\alpha_1,0) + (\alpha_2,0) - \infty_+ - \infty_-]$; and if $A = E_1 \times E_2$,  it means $P = \left((\alpha_1,0), (\alpha_2,0)\right)$. 

\begin{algorithm}[h!]
	\caption{\texttt{lift\_2\_torsion}}\label{algo:lift_2torsion}
	\begin{flushleft}
		\textbf{Input:} A principally polarised abelian surface $A$, a point $P \in A(k)[2]$, and an $m$-th order deformation $\A$.\\
		\textbf{Output:} The lift $\tilde{P} \in \A[2]$ of $P$. 
	\end{flushleft}
	\begin{algorithmic}[1]
		\State $\{\alpha_1,\alpha_2\} \gets P$
		\If{$A = \Jac(C)$ with $C:y^2=f(x)$}
		\State $\tilde{f} \gets $ a lift of $f$ with $\C:y^2 = \tilde{f}(x)$ and $\A = \Jac(\C)$
		\State $\tilde{f}_i \gets \tilde{f}$ for $i =1,2$
		\Else
		\State {$A = E_1 \times E_2$ with $E_1: y_1^2 = f_1(x_1)$ and $E_2: y_2^2 = f_2(x_2)$}
		\If{$\A = \E_1 \times \E_2$}
		\State $\tilde{f}_i \gets $ lift of $f_i$ with $\E_i: y_i^2 = \tilde{f}_i(x_i)$ for $i =1,2$
		\Else
		\State {$\A = \Jac(\C)$ with $\C: y^2 = \tilde{f}(x)$}
		\State $\tilde{f}_i \gets $ lift of $f_i$ with $\C: y_i^2 = \tilde{f}_i(x_i)$ for $i =1,2$
		\EndIf
		\EndIf
		\State $\tilde{\alpha}_i \gets \texttt{Newton\_lift}(\alpha_i, \tilde{f}_i)$ for $i =1,2$
		\State \Return $\{\tilde{\alpha}_1,\tilde{\alpha}_2\} \in \A[2]$
	\end{algorithmic}
	\end{algorithm}

\begin{lemma} \label{lem:lift-2tor}
	On input a point $P \in A[2]$, where $A$ is a principally polarized abelian surface over $k$,  and an $m$-th order deformation $\A$, Algorithm \ref{algo:lift_point} returns the unique lift $\tilde{P} \in \A[2]$ of $P$. The algorithm runs in time $O(1)$ over $R = k[\epsilon]/(\epsilon^{m+1})$.
\end{lemma}

\begin{proof}	
	First assume that $A$ is indecomposable, i.e. $A = \Jac(C)$ with $C:y^2 = f(x)$. Then $\A = \Jac(\C)$ and $\C:y^2 = \tilde{f}(x)$ is a deformation of $C$. In particular, lifting the torsion point $P$ to $\Jac(\C)$ consists in lifting the roots $\alpha_1,\alpha_2$ of $f$ to roots of the polynomial $\tilde{f} \in R[x]$. Note that  we need to choose an equation $y^2 = \tilde{f}(x)$ so that $\tilde{f} \equiv f \pmod{(\epsilon)}$.\footnote{In our applications (e.g. Algorithm \ref{algo:Richelot-chain}), this is automatically the case.}
	
	On the other hand, if $A = E \times E'$, there are two cases to consider. If $\A = \E \times \E'$ is also a product of elliptic curves, then the situation is similar as above. We simply lift roots $\alpha_1,\alpha_2$ to roots of the defining polynomials for $\E$ and $\E'$ respectively. 
	
	The case where $\A=\Jac(\C)$ is indecomposable, but $A = E_1 \times E_2$ is a product of elliptic curves is more subtle. For $i \in \{1,2\}$, we write $E_i: y_i^2 = f_i(x_i)$. Note that there necessarily exists a relation between the coordinates $(x_1,y_1)$ and $(x_2,y_2)$. In particular,  there exist polynomials $\tilde{f}_i$ (of degree $5$ or $6$) in $k[\epsilon]/(\epsilon^{m+1})[x]$ such that the reduction modulo $(\epsilon)$ is equal to $f_i$. Consequently, the root $\alpha_i$ can be lifted to a root of $\tilde{f}_i$ using \texttt{Newton\_lift}. Note that in order to interpret the resulting representation $\{\tilde{\alpha}_1,\tilde{\alpha}_2\} \in \A[2]$ of the $2$-torsion point correctly, one needs to consider the polynomials $\tilde{f}_1$ and $\tilde{f}_2$.
\end{proof}

\subsection{Lifting an isogeny diamond} \label{subsec:lift-diamond}

Given an isogeny diamond over a field $k$ as in Definition \ref{def:isogeny-diamond}, there are two ways to lift it to $R$. The straight-forward method is to lift the individual elliptic curve isogenies to isogenies over $R$. 
Here, we describe a different method (Algorithm \ref{algo:lift_isogeny_diamond}), where the product isogeny  induced by Kani's Lemma is directly lifted to $R$. The key ingredient for this algorithm is Corollary \ref{cor:chi10}.

\begin{algorithm}[h!]
	\caption{\texttt{lift\_isogeny\_diamond}}\label{algo:lift_isogeny_diamond}
	\begin{flushleft}
		\textbf{Input:} Elliptic curves $E,E_{a},E_b, E_{ab}$ which are the vertices of an isogeny diamond, and a group $K \subset E \times E_{ab}$ which is the kernel of the corresponding product isogeny.
		An $m$-th order deformation $\E$ of $E$.\\
		Assertion: $j(E), j(E_a), j(E_b), j(E_{ab}) \notin \{0,1728\}$.\\
		\textbf{Output:} The lifted isogeny diamond $(\E,\E_a,\E_b,\E_{ab})$ over $R$.
	\end{flushleft}

	\begin{algorithmic}[1]
		\State $\E_{ab} \gets E_{ab}$, $\tilde{K} \gets K$
		\linecomment{The precision is doubled at each step.}
		\For{$r \gets 1, \dots, \lceil\log_2(m+1)\rceil$}
		\State $j_{ab} = j(\E_{ab})$
		\State $R \gets k[\epsilon]/\left(\epsilon^{\min(2^r,m+1)}\right)$
		\linecomment{We compute $\chi_{10}$ for two different lifts of $E_{ab}$.}
		\For{$i = 0,1$}
		\State $j_i \gets R(j_{ab}) + i \cdot \epsilon^{2^{r-1}}$
		\State $\E_{i} \gets$ \texttt{deformation}$(E_{ab})$ with $j(E_{i}) = j_i$ over $R$.
		\State $K_i\gets$ \texttt{lift}$(\tilde{K}) \subset (\E \times \E_i)[N]$.
		\State $A_i \gets (\E \times \E_i)/K_i$.
		\State $\delta_i \gets \chi_{10}(A_i)$
		\EndFor
		\linecomment{The correct lift of $\E_{ab}$ is deduced from $\delta_0,\delta_1$.}
		\State $j_{ab} \gets j_{ab} - \frac{\delta_0}{\delta_1 - \delta_0}$
		\State $\E_{ab} \gets$ \texttt{deformation}$(E_{ab})$ with $j(\E_{ab}) = j_{ab}$
		\State $\tilde{K} \gets$ \texttt{lift}$(\tilde{K}) \subset (E \times E'')[N]$
		\EndFor
		\linecomment{We compute the lifted product isogeny.}
		\State $\E_a \times \E_b \gets (\E \times \E_{ab})/\tilde{K}$. \label{line:last-step}
		\State \Return $(\E_a,\E_b,\E_{ab})$
	\end{algorithmic}
\end{algorithm}

\begin{proposition}
	Let  $E,E_{a},E_b, E_{ab}$ be elliptic curves which are the vertices of an isogeny diamond, and let $K \subset E \times E_{ab}$ the kernel of the product isogeny induced by Kani's Lemma. Further assume that none of the elliptic curves have extra automorphisms, that is we assume $j(E), j(E_a), j(E_b), j(E_{ab}) \notin \{0,1728\}$. On input this data together with an $m$-th order deformation $\E$ of $E$, Algorithm \ref{algo:lift_isogeny_diamond} outputs the deformations $(\E_a,\E_b,\E_{ab})$ such that $F$ lifts to an isogeny $\tilde{F}:\E \times \E_{ab} \to \E_a \times \E_b$.
\end{proposition}

\begin{proof}
	The main part of the algorithm consists in computing the correct deformation of $E_{ab}$. Once this is done, we can lift the kernel $K$ to a subgroup of $\E \times \E_{ab}$, and then $\E_a \times \E_b$ is computed as the codomain of the product isogeny with this kernel (Line \ref{line:last-step}). 

	To understand  the main part, recall from Corollary \ref{cor:chi10} that there exists a polynomial $h \in k[x_1,x_2]$ with the property that 
	\[
	h(j(\E) -j(E), j(\E') - j(E_{ab})) = \chi_{10}(\A),
	\] 
	for any arbitrary deformation $\E'$ of $E_{ab}$ and where $\A$ is the corresponding codomain of the lifted isogeny $F:\E \times \E' \to \A$. Our goal is to find the correct deformation $\E' = \E_{ab}$ so that $\A$ is again decomposable, i.e. $\chi_{10} = 0$.  If we knew the polynomial $h$, we could simply apply the Newton method to find the correct value for $j(\E_{ab})$. Since this is not the case, we use an interpolation based approach, where we evaluate $h$ at two different values of $j(\E')$. Similar as in Newton's method, this approach lets us double the precision at each step.  
	
	To describe the idea in more detail, assume we are at Step $r$ of the algorithm. Let $E_{ab,{r-1}}$ be the correct deformation of order $2^{r-1}-1$, that is 
	\[
	h(j(\E) -j(E), j(E_{ab,r-1}) - j(E_{ab})) \equiv 0 \pmod{(\epsilon^{2^{r-1}})}.
	\]
	Now, we compute two different deformations of $E_{ab}$ of order $2^r-1$, and denote them by $E_0$ and $E_1$. More precisely, we choose the lifts with $j$-invariants 
	\[
	j(E_0) = j(E_{ab,r-1}), \quad j(E_1) = j(E_{ab,r-1}) + \epsilon^{2^{r-1}}.
	\]
	Note that we can compute arbitrary deformations, since $j(E_{ab}) \neq 0 ,1728$ by assumption (cf. Example \ref{exa:deformation-parameters}).
	For both cases $i = 0,1$, we lift the kernel $K$ to a subgroup $K_i \subset \E \times E_i$ using the method \texttt{lift\_point} described in Algorithm \ref{algo:lift_point}. Then we compute the codomains  $A_i = (\E \times E_i)/K_i$ of the resulting isogenies, and $\delta_i = \chi_{10}(A_i)$. With $j(E_{ab,r}) =j(E_{ab,r-1}) -  \delta_0/(\delta_1-\delta_0)$, we find that  
	\[
	h(j(\E) -j(E), j(E_{ab,r}) - j(E_{ab})) \equiv 0 \pmod{(\epsilon^{2^{r}})}.
	\]
	After computing the data corresponding to this deformation, one can proceed to Step $r+1$. 
\end{proof}

\subsection{Lifting (smooth-degree) isogenies} \label{subsec:lift-smooth-isogeny}

Given an isogeny $f:A\to A'$ and a deformation $\A$ of $A$, Proposition \ref{prop:lift-isogenies}  shows that there exists a unique lift $\tilde{f}:\A \to \A'$. 
Let $G \subset A[N]$ be the kernel of $f$. The straight-forward approach for lifting $f$ consists in lifting the kernel generators, for instance using Algorithms \ref{algo:lift_point}, in order to obtain a lifted kernel $\tilde{G} \subset \A[N]$. Then the isogeny can be computed using standard algorithms.
This is the approach used in Algorithm \ref{algo:lift_isogeny_diamond}. 

However, if $N$ is composite the above approach is not optimal. For instance, in the case $N = 2^n$, an isogeny can be lifted in time $O(n)$ (provided some auxiliary data from the original isogeny), while the naive method has complexity $O(n\log(n))$. To keep everything explicit, we restrict to the case of $(2^n,2^n)$-isogenies. 

\begin{algorithm}
	\caption{\texttt{lift\_2\_2\_chain}}\label{algo:Richelot-chain}
	\begin{flushleft}
		\textbf{Input:} A $(2^n,2^n)$-isogeny $f = f_n \circ \dots \circ f_1 : A_0 \to A_n$ over $k$, the data $\ker(f_i) = \langle P_{1,i}, P_{2,i}\rangle$ for each $i$, and an $m$-th order deformation $\A_0$ of $A_0$.\\
		\textbf{Output:} The deformation $\A_n$ so that $f$ lifts to an isogeny $\tilde{f}:\A_0 \to \A_n$. 
	\end{flushleft}
	\begin{algorithmic}[1]	
		\For{$i = 1, \dots, n$}
		\State $\tilde{P}_{j,i} \gets$ \texttt{lift\_2\_torsion} $(A_{i-1},P_{j,i},\A_{i-1})$ for $j = 1,2$
		\State $\A_i~~ \gets \texttt{2\_2\_isogeny}(\A_{i-1}, \langle \tilde{P}_{1,i}, \tilde{P}_{2,i} \rangle)$
		\EndFor\\
		\Return $\A_n$
	\end{algorithmic}
\end{algorithm}

\begin{lemma}
	Given a $(2^n,2^n)$-isogeny $f = f_n \circ \dots \circ f_1 : \A_0 \to \A_n$ over a field  $k$ together with the data $\ker(f_i) = \langle P_{1,i}, P_{2,i}\rangle$ for each $i \in \{1,\dots,n\}$, and an $m$-th order deformation $\A_0$ of $A_0$, Algorithm \ref{algo:Richelot-chain} outputs a deformation $\A_n$ of $A_n$ so that  $f$ lifts to an isogeny $\tilde{f}:\A_0 \to \A_n$. The algorithm runs in time $O(n)$ over $R = k[\epsilon]/(\epsilon^{m+1})$.
\end{lemma}

\begin{proof}
	The isogeny chain computed in Algorithm \ref{algo:Richelot-chain} is correct, since at each step the kernel generators are lifted correctly as per Lemma \ref{lem:lift-2tor}. The method \texttt{2\_2\_isogeny} is used as a black box. It is only important that it runs in time $O(1)$ over $R$.	
\end{proof}

\begin{remark}
  \label{rem:radical}
  An even faster method to lift a smooth degree isogeny is to use modular
  correspondences. For instance, a $2^n$-isogeny over $k$ is given by a
  sequence of modular points which are solutions of the modular
  correspondance $\correspbis$, and lifting the isogeny corresponds to
  lifting this sequence of points.
  In practice, since we want to be able to evaluate modular forms, we need
  to work with normalised isogenies, hence use a version of $\correspbis$ which
  keeps track of our differentials.

  For instance, Richelot isogenies are well suited to write explicit modular
  equations. Explicit formulas are given by \enquote{multiradical isogeny}
  formulas \cite[\S~4.2]{castryck2021multiradical} which describe the modular correspondence
  $\correspbis: \Am_2(\Gamma_1(4)) \to  \Am_2(\Gamma'_1(4)) \to \Am_2(\Gamma_1(2)) \times
  \Am_2(\Gamma_1(2))$ (see \cite[\S~3.1]{castryck2021multiradical} for the
  notations).
  The theta duplication formula is also naturally $2$-radical in
  dimension~$2$.
  Using one of these radical $2$-isogeny formulas, lifting each such $2$-isogeny
  in dimension~$2$  amounts to lifting (via Newton) three square roots. In our implementation, we follow this approach.
  It has the advantage that it bypasses the need to lift torsion
  points and evaluate isogeny formulas on these lifted points.
  Another potential (constant-time) speed-up could be obtained by using $4$-radical isogeny formulas instead.
\end{remark}

\subsection{Deforming a general isogeny}
\label{subsec:generaldeformation}
In this section we prove the general statement from \cref{thm:main2}.
We are given an efficient representation of an isogeny $f: E_1 \to E_2$,
and we want to lift it to $\tilde{f}: \E_1 \to \E_2$
for some given $m$-th order deformation $\E_1/R$.

\smallskip
\textbf{Initialisation}: We fix $N>\ell$ with $N$ prime to $\ell$ and $p$, and such that the $N$-torsion is accessible
(which means that its Sylow subgroups lie in small extensions of the base
field), for instance $N$ is powersmooth or $N=2^n$ if the $2^n$-torsion of $E$
 is rational.
 We write $N-\ell = a^2+b^2+c^2+d^2$, and consider $\alpha$, the $4 \times 4$ quaternion matrix with norm $N-\ell$. Then $\alpha$ induces an endomorphism on $E_1^4$ and $E_2^4$.
 We can construct the isogeny diamond
	\begin{center}
		\begin{tikzcd}
			E_1^4  \arrow[r, "f"] \arrow[d, "\alpha"] & E_2^4  \arrow[d, "\alpha"] \\
			E_1^4  \arrow[r, "f"]&  E_2^4.
		\end{tikzcd}
	\end{center}
which allows us to embed $f$ into an $8$-dimensional $N$-endomorphism
$F: E_1^4 \times E_2^4 \to {E}_1^4 \times {E}_2^4$.
By assumption, we have an efficient representation of $f$, hence we can evaluate
it on the $N$-torsion in order to compute the HD representation $F$  of $f$.
See \cite{EC23:Robert} for more details.

Since $N$ is taken to be smooth, we can decompose $F$ as a product of
small degree isogenies. In theory, $F$ can be computed in the theta
model, using the isogeny algorithm of \cite{DRfastisogenies},
working with level~$4$ theta functions, and representing our
intermediate abelian varieties by their theta constants. The codomain of
$F$ is equal to $E_1^4 \times E_2^4$ but it may have different theta
constants than the domain (because the level $\Gamma(4,8)$-theta structure
needs not be preserved), so we apply a symplectic change of basis at the
end to get matching theta constants. A way to compute this matrix is
described in \cite[Appendix~F]{sqisignhd}.

\smallskip
\textbf{Newton iterations}:
We then proceed by a Newton iteration, doubling the precision~$m$ at each
stage. For simplicity we explain how to go from $m=1$ to $m=2$ here, the
general case being similar.

Let $\tilde{f}: \E_1 \to \E_2$ be the deformation of $f$ to
$\E_1$. The isogeny diamond above certainly deforms (with the same
matrix $\alpha$), and so $F$ deforms to an endomorphism $\tilde{F}$ of $\E_1^4 \times  \E_2^4$.
On the other hand, if we take an arbitrary lift $\E_2'$ of $E_2$ to
$R$, the corresponding deformation $\tilde{F}'$ is only an isogeny.
In fact, if $\tilde{F}'$ is an endomorphism, then it is given by a
matrix, which reduces to the matrix giving $F$ (by unicity of
deformations), hence contains the deformation $\tilde{f}$.

In summary, $\E_2'$ is the correct codomain of the deformation of
$f$ if and only if
the codomain of $\tilde{F}$ is equal to the domain. Using modular
invariants $J$ (e.g. the theta constants), this can be tested via the
equality $J(\E_1^4 \times
\E_2'^4)=J(\mathrm{codomain}{(\tilde{F})})$.
By the same arguments as in \cref{sec:deformations-theory}, in particular \cref{cor:chi10}, the left and
right member of this equality are analytic in terms of the deformation
parameter $\lambda$ of $\E_2$ (i.e. are given by power series).
Since we are doing a Newton iteration, everything becomes linear, so we just
need to interpolate between two different evaluations of the deformation
parameter $\lambda$ to solve the equation.
For each of these two choices of $\lambda$, we need to compute the
deformation of $\tilde{F}$ to recover its codomain.

\textit{Computing $\tilde{F}$}:
We proceed as in \cref{subsec:lift-smooth-isogeny}.
Since $F$ is decomposed into a product of small isogenies over $\Fq$, we
deform these isogenies step by step.
The easiest way is to deform the points in the kernel and apply the theta
isogeny algorithm.

\textit{Deforming a point $P$ of $N$-torsion}:
We also proceed by a Newton iteration.
Division polynomials are not easy to compute in higher dimension.
Instead, we just rely on the fact that there is a unique deformation
$\tilde{P}$ of $P$ which is still of $n$-torsion, and that taking an
arbitrary deformation $\tilde{P}'$ we can efficiently compute $n \cdot 
\tilde{P}'$. In other words, it is easy  to evaluate the $n$-division
polynomials, which is enough for the Newton iteration.

\textbf{Complexity}:
Since we double the precision $m$ at each step, 
and the arithmetic of $R$ is super-linear
in terms of $m$, the last Newton step is dominating. 

The dominant step is in the computation of the deformation $\tilde{F}$;
it consists in deforming $O(\log \ell)$ isogenies of degree $O(\log \ell)$
in the worst case (e.g. when $N$ is powersmooth).
The algorithm cost is thus polynomial in $\log \ell$ arithmetic operations (for deforming the generators of the kernels and then computing the isogeny) in
$R$ at precision~$m$.

In the best case, we can take $N=2^n$ and the $N$-torsion is rational in
$E_1$, so we just need to deform $O(\log \ell)$ $2$-isogenies.

%% file: modular-Fp.tex
\section{Computing modular polynomials over finite fields} \label{sec:modular-Fp}

In this section, we present algorithms for computing the modular polynomial $\varphi_\ell(X,Y)$ over $\FF_p$.\footnote{We use the notation $\varphi_\ell$ for the modular polynomial over a finite field, and $\Phi_\ell$ for the modular polynomial with coefficients in $\ZZ$.} We describe very explicit versions of the algorithm for a family of primes $p$ depending on $\ell$. These methods will be used in our implementation for computing the modular polynomial over $\QQ$. In addition, \cref{subsec:generalcase} discusses the generalization to arbitrary primes.

In all cases, the main ingredient for our procedure is an algorithm to lift isogeny diamonds over $\FF_p$ to isogeny diamonds over $\FF_p[[\epsilon]]$ with precision at least $\ell+2$. The case distinction is necessary, since we construct $(\ell,d)$-isogeny diamonds with $\ell+d$ smooth, which depends on the value of $\ell$.

\subsection{Suitable primes} \label{subsec:suitable-primes}
Depending on  $\ell \pmod{4}$, we define a set of primes $\P_{\ell}$ for which our algorithm can compute the modular polynomial over $\FF_p$. 
For an odd prime $\ell$, we set
\begin{align*}
	\P_\ell = \{ p > 11 \text{ prime}: \exists~ n,a,b \text{ with } 2^n - c_\ell \cdot \ell = a^2 + 4b^2 \text{ and } 2^n \cdot c_\ell \cdot \ell \mid p+1\},
\end{align*}
where 
\[
c_\ell = \begin{cases}
	1 & \text{if } \ell \equiv 3 \pmod{4}, \\
	3 & \text{if } \ell \equiv 1 \pmod{4}.
\end{cases}
\]

\begin{heuristic} \label{heuristic:suitable-n}
	Let $\ell$ be a positive integer and $c_\ell$ as defined above. Then we expect
	\[
	\# \{n \in \NN \mid 2^n - c_\ell \cdot \ell = a^2+4b^2, \; n \leq x \} 
	\approx \frac{x}{\sqrt{x}}.
	\]

\end{heuristic}

Under Heuristic \ref{heuristic:suitable-n}, we expect to find a value $n \in \O(\log(\ell))$. The intuition behind this heuristic and our choice for $c_\ell$ is explained in the remark below.

\begin{remark}
	Recall that an element $z \in \NN$ can be written as a sum of two squares, $z = a^2+b^2$, if and only if its prime decomposition $z = \prod_i p_i^{k_i}$ contains no prime factor $p_i = 3 \pmod{4}$ with $k_i$ odd. In particular, an element $z \equiv 3 \pmod{4}$ cannot be written as a sum of two squares.  Now $c_\ell \in \{1,3\}$ is chosen so that this necessary congruence condition is satisfied for all $z = 2^n - c_\ell \cdot \ell$. 
	
	Note that in our setting, we need to write $z = a^2 + 4b^2$. However, we are always in the case that $z$ is odd, hence this is equivalent to $z$ being a sum of two squares.	
	The fraction of elements smaller than some $x$ that can be written as a sum of two squares is known to be in  $O\left(1/\sqrt{\log(x)}\right)$, hence
	\[
	\# \{z = a^2 + 4b^2 \mid z \leq x\} \approx \frac{x}{\sqrt{\log(x)}}.
	\]
	Essentially, Heuristic \ref{heuristic:suitable-n} means that we expect the elements of the form $a^2 + 4b^2$ to be uniformly distributed among the numbers of the form  $2^n - c_\ell \cdot \ell$ for varying $n \in \NN$.
\end{remark}

Under Heuristic \ref{heuristic:suitable-n}, we obtain that there are $\O(\ell)$ primes $p \in P_\ell$ that have about the same logarithmic size as $\ell$, i.e. $\log(p)$ in $\O(\log(\ell))$. A more precise estimate is provided in the lemma below.

\begin{lemma} \label{lem:primes-in-Pl}
	Let $\ell$ be a positive integer and $\P_{\ell}$ as defined above. 
	Then 
	\[
	\# \{ p \in P_\ell \mid p \leq x\} \gtrapprox \frac{x}{2^n  \ell  \log(x)}.
	\]
	More precisely, the $m$-th prime in $P_{\ell}$ is bounded by
	$(2^n \ell)^L m \log(m)^2$, where $L \leq 5$ is Linnik's constant.
	If Heuristic \ref{heuristic:suitable-n} holds,  the bound is
	$(\ell)^{L'} m \log(m)^2$, for some $L'$.
\end{lemma}

\begin{proof}
	We look at the smallest integer so that there exist integers $a,b$ with $2^n - c_\ell \cdot \ell = a^2 + 4b^2$. For this triple $(n,a,b)$, we can sieve through the primes to find those which satisfy 
	\[
	p \equiv -1 \pmod{2^n \cdot c_\ell \cdot \ell}.
	\] 
	It follows from the {\em prime number theorem for arithmetic progressions} that the proportion of primes of this form is $1/\varphi(2^n \cdot c_\ell \cdot \ell) \approx 1/\ell^2$ (with $\varphi$ denoting Euler's totient function). 

        For the more precise statement, we need an upper bound on the
        smallest suitable prime.
We need to find primes congruent to $-1$ modulo 
$A=2^n \cdot \ell$. By Linnik's theorem \cite[Corollary~18.8]{iwaniec2021analytic}, the $m$-th such prime is bounded by $A^L m
\log(m)^2$ where $L\leq 5$ is Linnik's constant.
Under Heuristic \ref{heuristic:suitable-n}, we have $2^n \leq
\ell^{\const}$ for some constant $\const$, which concludes the proof.
\end{proof}

\subsection{The case $\ell \equiv 3 \pmod{4}$} \label{subsec:modular-poly-a}

Here, we present an algorithm for computing the modular polynomial $\varphi_\ell$ over $\FF_p$ when $\ell \equiv 3 \pmod{4}$ for primes $p$ in the set $\mathcal{P}_\ell$.

\begin{algorithm}[h!]
	\caption{\texttt{modular\_polynomial\_modp} ~ $\ell \equiv 3 \pmod{4}$ }\label{algo:modp}
	\begin{flushleft}
	\textbf{Input:} A prime  $\ell \equiv 3 \pmod{4}$, and a prime $p\in \mathcal{P}_\ell$.\\
	\textbf{Output:} The modular polynomial $\varphi_\ell(X,Y) \in \FF_p[X,Y]$.
	\end{flushleft}
	\begin{algorithmic}[1]
		\linecomment{Setup}
		\State Set $E: y^2 = x^3 + 6x^2 + x$ over $\FF_{p^2}$ 
		\State $\iota \gets \iota \in \End(E)$ with $\iota \circ {\iota} = [-4]$
		\State $\gamma \gets [a] + [b] \iota$, where $2^n - \ell = a^2 + 4b^2$ for some $n$. \label{line:definition-gamma}
		\State Compute $P_{2^n}, Q_{2^n}$ with $E[2^n] = \langle P_{2^n}, Q_{2^n} \rangle$
		\State Compute $P_\ell, Q_\ell$ with $E[\ell] = \langle P_\ell, Q_\ell \rangle$
		\State $\tilde{j} \gets j(E) + \epsilon \in \FF_{p^2}[\epsilon]/(\epsilon^{\ell+2})$
		\State $\E \gets $ elliptic curve with $j(\E) = \tilde{j}$
		\linecomment{computing and lifting all $\ell$-isogenies}
		\For{$k \gets 0, \dots, \ell$}
		\If{$k = \ell$}
		\State $P_k \gets Q_\ell$
		\Else
		\State $P_k \gets P_\ell + k \cdot Q_\ell$
		\EndIf
		\State $E_k \gets E/\langle P_k\rangle$ with $f_k: E \to E_k$
		\linecomment{constructing the $(\ell,2^n-\ell)$-isogeny diamond}
		\State $E_k' \gets E/\langle \gamma(P_k)\rangle$ with $f_k': E \to E_k'$
		\State $K \gets \langle (\hat{\gamma}(P_{2^n}), f_k'(P_{2^n})), (\hat{\gamma}(Q_{2^n}), f_k'(Q_{2^n})) \rangle$
		\State $(\E,\E_k,\E',\E_k') \gets $ \texttt{lift\_isogeny\_diamond}$(E,E_k,E,E_k',K, \E)$
		\State $\tilde{j_k} \gets j(\E_k)$
		\EndFor
		\linecomment{final step}
		\State $\varphi \gets \prod_{k=0}^{\ell}(Y - \tilde{j_k})(\epsilon = X - j(E)) \in \FF_p[X,Y]$
		\State \Return $\varphi$
	\end{algorithmic}
\end{algorithm}

\begin{theorem}\label{thm:modular-polynomial-a}
	On input a prime  $\ell \equiv 3 \pmod{4}$ and a prime $p \in \mathcal{P}_\ell$, Algorithm \ref{algo:modp} returns the modular polynomial $\varphi_\ell$ over $\FF_p$. The algorithm runs in  \[
	O(\log p + \ell^2) \cdot \mathsf{M}(\FF_{p^2}) + O(n\ell) \cdot \mathsf{M}(R)
        + O(\log \ell) \mathsf{M}(R, \ell),
	\]
	with $R = \FF_{p^2}[\epsilon]/(\epsilon^{\ell+2})$, and $\mathsf{M}(\cdot)$ as defined in Subsection \ref{subsec:arithmetic-in-R}.
\end{theorem}

\begin{proof}
First, the algorithm sets $E: y^2 = x^3 + 6x^2 +x$ over $\FF_{p^2}$. This curve is $2$-isogenous to the elliptic curve with $j$-invariant $0$, hence there exists an endomorphism $\iota: E \to E$ with the property $\iota^2 = [-4]$. Since $p \in P_\ell$, there exist values $n,a,b$ so that $2^n-\ell = a^2+4b^2$, and the algorithm sets $\gamma = [a]+[b]\iota$ which is an endomorphism of degree $2^n-\ell$. 
Further, a basis $(P_{2^n},Q_{2^n})$ for $E[2^n]$ and a basis $(P_\ell,Q_\ell)$ for $E[\ell]$ are computed. Note that by assumption $p \equiv 3 \pmod{4}$, hence the elliptic curve $E$ is supersingular and has cardinality $(p+1)^2$. As a consequence, the $\ell\cdot 2^n$-torsion is $\FF_{p^2}$-rational and it is not necessary to work in field extensions. The cost of the setup is dominated by the cost for computing the $N$-torsion bases for $N \in \{\ell,2^n\}$. This  can be done by sampling two points $P,Q \in E(\FF_{p^2})$, computing $P_N = \frac{p+1}{N} \cdot P,\; Q_N = \frac{p+1}{N}\cdot Q$ and the Weil pairing $e_N(P_N,Q_N)$. This is repeated until $e_N(P_N,Q_N)$ has order $N$.\footnote{There are better methods to compute a basis in practice. In particular, in the case of $N=2^n$, one should use the methods developed in the framework of SIDH key compression \cite{costello2017efficient}.} This step costs 
\begin{equation} \label{eq:setup1}
 (\log(p) + \log \ell ) \cdot \mathsf{M}(\FF_{p^2})
\end{equation}

In the last line of the setup, we compute the deformation $\E$ of $E$ with $j$-invariant $j(\E) = j(E) + \epsilon$. Note that $j(E) = (2\cdot 3 \cdot 11)^3 \notin \{0,1728\}$ for $p >11$, hence such a deformation exists, and it can be computed in time $O(1)$ in $R = \FF_{p^2}[\epsilon]/(\epsilon^{\ell+2})$. This means the complexity of this step is simply
\begin{equation} \label{eq:setup2}
	O(1) \cdot \mathsf{M}(R).
\end{equation}

The main part of the algorithm consists of computing the $\ell+1$ different $\ell$-isogenies emanating from the elliptic curve $E$; and lifting these to isogenies emanating from the elliptic curve $\E$ with $j$-invariant $\tilde{j} = j+\epsilon$. Using the square-root V\'elu Algorithm, the isogeny computations over the ground field can be done in $\Otilde(\ell \sqrt{\ell})$ $\FF_{p^2}$-multiplications. For the overall analysis it is sufficient to use the standard V\'elu formulas, which results in 
\begin{equation} \label{eq:ell-isogenies}
  	O(\ell^2) \cdot \mathsf{M}(\FF_{p^2}).
\end{equation}

To explain the lifting step in more detail, assume that we are at step $k$ and want to lift the $\ell$-isogeny $f_k:E \to E_k$ to an isogeny with domain $\E$. 
To this end, we construct an $(\ell, 2^n-\ell)$-isogeny diamond as depicted below.
\begin{center}
	\begin{tikzcd}
		E \arrow[r, "f_k"] \arrow[d, "\gamma"] & E_k \arrow[d, "\gamma'"] \\
		E \arrow[r, "f_k'"]& E_k'
	\end{tikzcd}
\end{center}
This gives rise to a product isogeny $F: E \times E_k' \to E \times E_k$ with kernel 
\[
K = \langle (\hat{\gamma}(P_{2^n}), f_k'(P_{2^n})), (\hat{\gamma}(Q_{2^n}), f_k'(Q_{2^n})) \rangle.
\]
On input the vertices of the isogeny diamond $E,E_k,E,E_k'$ and the lift $\E$ together with the kernel $K$, the method \texttt{lift\_isogeny\_diamond} (Algorithm \ref{algo:lift_isogeny_diamond}) outputs the vertices of the lifted isogeny diamond $(\E,\E_k,\E',\E_k')$. In particular, this contains the lift $\E_k$ of $E_k$ such that $f_k$ lifts to an isogeny $\tilde{f_k}:\E \to \E_k'$ over $\FF_{p^2}[\epsilon]/(\epsilon^{\ell+2})$. We save its $j$-invariant $\tilde{j_k}$. 
Note that since we are lifting a $(2^n,2^n)$-chain here, we can use an improved lifting strategy in \texttt{lift\_isogeny\_diamond} as described in Subsection \ref{subsec:lift-smooth-isogeny}. For each $\ell$-isogeny this gives us complexity $O(n)$ in $R$, hence overall the lifting step for all $\ell$-isogenies costs
\begin{equation} \label{eq:lifting}
  O(n \ell) \cdot  \mathsf{M}(R)
\end{equation}

In the final step, the product 
\[
	\varphi_\ell(j(E) + \epsilon, Y) =  \prod_{k=0}^\ell (Y - \tilde{j_k}) \in R[Y].
\]
is computed and evaluated at $\epsilon = X - j(E)$. This yields the modular polynomial $\varphi_\ell(X,Y)$ (modulo $(X - j(E))^{\ell+2}$). Given that $\varphi_\ell(X,Y)$ has degree $\ell+1$, we may ignore the modulus and we conclude that the output is the modular polynomial $\varphi_\ell(X,Y)$ in $\FF_p$.  The evaluation of the product can be done via a product-tree strategy,
which costs $O(\log \ell) \mathsf{M}(R, \ell)$.
In our setting, the complexity for this multiplication is simply given by 
\begin{equation} \label{eq:final-product}
  O(\log \ell) \mathsf{M}(R, \ell)
\end{equation}

Taking the sum over the complexities from Equations \ref{eq:setup1} - \ref{eq:final-product}, we obtain a complexity of
\[
O(\log p + \ell^2) \cdot \mathsf{M}(\FF_{p^2}) + O(n\,\ell) \cdot \mathsf{M}(R)
+ O(\log \ell) \mathsf{M}(R, \ell).
\]
\end{proof}

\subsection{Generalization to $\ell \equiv 1 \pmod{4}$} \label{subsec:modular-poly-b}

In Algorithm \ref{algo:modp} we described a method to compute the modular polynomial $\varphi_\ell$ in $\FF_p[X,Y]$ for primes in $P_\ell$, when $\ell \equiv 3 \pmod{4}$. In this part, we sketch a modification of the algorithm which allows us to  compute the modular polynomial for $\ell \equiv 1 \pmod{4}$. 

The different treatment stems from the fact that in this case, there do not exist integers $a,b,n$ so that $2^n-\ell =  a^2 + 4b^2$. Instead we consider the equation $2^n - 3 \cdot \ell =  a^2 + 4b^2$, and extend the $\ell$-isogenies by an auxiliary $3$-isogeny.

\begin{theorem}\label{thm:modular-polynomial-b}
	There exists a modification of Algorithm \ref{algo:modp} which on input a prime  $\ell \equiv 1 \pmod{4}$ and a prime $p \in \mathcal{P}_\ell$, returns the modular polynomial $\varphi_\ell$ over $\FF_p$. The algorithm runs in  \[
	O(\log p + \ell^2) \cdot \mathsf{M}(\FF_{p^2}) + O(n\ell) \cdot \mathsf{M}(R)
	+O(\log \ell) \cdot \mathsf{M}(R, \ell),
	\]
	with $R = \FF_{p^2}(\epsilon)/(\epsilon^{\ell+2})$, and $\mathsf{M}(\cdot)$ as defined in Subsection \ref{subsec:arithmetic-in-R}.
\end{theorem}

\begin{proof}
	
	We outline the modifications of Algorithm \ref{algo:modp}. It will be clear that these changes have no effect on the  asymptotic complexity of the algorithm. 
	
	In the setup, we choose an endomorphism $\gamma$  as $\gamma = [a] + [b]\iota$ with $2^n-3 \cdot \ell = a^2+4b^2$.  And we also compute an auxiliary  degree-$3$ isogeny $g:E \to C_0$.

	The main change occurs in the construction of the isogeny diamond. In order to make $\gamma$ and the $\ell$-isogeny $f_k$ fit into an isogeny diamond (with smooth induced product isogeny), it is necessary to expand $f_k$ by the auxiliary isogeny $g$ to obtain the following configuration.
	\begin{equation}
		\label{eq:large-isogeny-diamond}
		\begin{tikzcd}
			C_0 \arrow[d, "\gamma_0"]& E \arrow[r, "f_k"] \arrow[l, "g"] \arrow[d, "\gamma"] & E_k \arrow[d, "\gamma'"] \\
			C_1 & E \arrow[r, "f_k'"] \arrow[l, "g'"] & E_k'
		\end{tikzcd}
	\end{equation}
	
	The outer square is a $(3\cdot \ell, 2^n-3 \cdot \ell)$-isogeny diamond and induces the product isogeny $F: C_0 \times E_k' \to C_1 \times E_k$ with kernel 
	\[
	K = \left\langle \left(\hat{\gamma_0}(P_{2^n}),f_k'\circ\hat{g'}(P_{2^n})\right), \,\left(\hat{\gamma_0}(Q_{2^n}),f_k'\circ\hat{g'}(Q_{2^n})\right) \right\rangle,
	\]
	where $(P_{2^n},Q_{2^n})$ is a basis for $C_1[2^n]$. 
	
	Recall that the goal is to lift the $\ell$-isogeny $f_k$, for some fixed lift $\E$ of $E$. For this purpose, one first computes the lift $\tilde{g}: \E \to \C_0$ explicitly.
	Then we call \texttt{lift\_isogeny\_diamond} on input $(C_0,E_k,C_1,E_k', K, \C_0)$.  The output contains the lift $\E_k$, and the algorithm proceeds as in the case $\ell \equiv 3 \pmod{4}$.
\end{proof}

The modified version of Algorithm \ref{algo:modp} which is outlined in the proof above, is presented in Appendix \ref{sec:variants} (Algorithm \ref{algo:modp_3}).

\subsection{The general case}
\label{subsec:generalcase}
In this section, we give a quick overview of the proof of \cref{cor:main}.
If we start with an HD representation of the $\ell$-isogenies starting from
$E_0$, we can deform them to $k[\epsilon]/(\epsilon^{m+1})$ using
\cref{thm:main2} as explained in \cref{subsec:generaldeformation}.
When $m>\ell$, this is enough to reconstruct $\Phi_{\ell} \mod p$ in quasi-linear time as in the proof of \cref{thm:modular-polynomial-a}.

In the previous two subsections, we provided algorithms to compute the modular polynomial $\varphi_\ell$ over a prime field $\FF_p$, where $p$ was an element in $\P_{\ell}$. The number of primes of this form relies on Heuristic \ref{heuristic:suitable-n} which describes the size of $n$, for which $2^n -\ell$ can be written as a sum of two squares. Using a dimension~$8$ embedding instead of a dimension~$2$ embedding, as explained in \cref{subsec:generaldeformation},
we generalize the techniques to a larger set of primes by expressing $2^n -\ell$ as a sum of four squares. More explicitly, this results in an algorithm 
for all primes $p$ in 
\[
\P^*_\ell = \{p > 11 \text{ prime : } 2^n \cdot \ell \mid p+1, \text{ where } n = \lceil\log_2(\ell)\rceil\}
\]
of complexity $O(\log^2 p + \ell^2 \log p)$.

%% file: modular-general.tex
\section{Computing the modular polynomial} \label{sec:modular}

Here, we present our algorithm for computing the modular polynomial $\Phi_\ell \in \ZZ[X,Y]$ of elliptic curves. The procedure is summarized in Algorithm \ref{algo:full}. It is a CRT based approach, where the modular polynomial is computed modulo many small primes using Algorithm \ref{algo:modp} or its modification described in Subsection \ref{subsec:modular-poly-b}.  The runtime of the algorithm relies on Heuristic \ref{heuristic:suitable-n}, but we prove that there also exists an unconditional version with the same asymptotic runtime. 

\begin{algorithm}
	\caption{\texttt{modular\_polynomial}}\label{algo:full}
	\begin{flushleft}
	\textbf{Input:} An odd prime $\ell$.\\
	\textbf{Output:} The modular polynomial $\Phi_\ell(X,Y) \in \ZZ[X,Y]$.
	\end{flushleft}
	\begin{algorithmic}[1]
		\State $B \gets 6 \ell \log(\ell) + 16 \ell + \min{(2\ell,14\sqrt{\ell}{\log{\ell}})} + \log(2)$ 
		\State $P \gets 1$, $\Phi_\ell \gets 0$.
		\State $c_\ell \gets -\ell \pmod{4}$.
		\State $n \gets \min\{n \mid \exists~ a,b : 2^n - c_\ell\cdot \ell = a^2 + 4b^2\}$ \label{line:choose-n}
		\While{$P < \exp(B)$}
		\State $p \gets $  next prime with $2^n \cdot c_\ell \cdot \ell \mid  p + 1$
		\State $\varphi_\ell \gets$ \texttt{modular\_polynomial\_modp}$(\ell,p) \in \FF_p[X,Y]$
		\State $P \gets P \cdot p$
		\State $\Phi_\ell \gets$ \texttt{CRT}$(\Phi_\ell, \varphi) \in \ZZ/P\ZZ[X,Y]$
		\EndWhile            
		\State \Return $\Phi$
	\end{algorithmic}     
\end{algorithm}        
                       
\begin{theorem} \label{thm:modular-polynomial}
	On input an odd prime $\ell$, Algorithm \ref{algo:full} computes the modular polynomial $\Phi_\ell \in \ZZ[X,Y]$. Under Heuristic \ref{heuristic:suitable-n}, the algorithm runs in time 
	\[
	O(\ell^3 \log^3 \ell \log\log \ell).
	\]
\end{theorem}

\begin{proof}
	The bound
	\[
	B = 6 \ell \log(\ell) + 16 \ell + \min{(2\ell,14\sqrt{\ell}{\log{\ell}})} + \log(2)
	\] 
	in the first line of the algorithm is equal to the bound on the logarithmic height of the coefficients of $\Phi_\ell$ (see Eq. \ref{eq:bound}) plus $\log(2)$. Consequently, the coefficients of $\Phi_\ell$ are uniquely determined by their residues modulo $\exp(B)$.
	
	In Line \ref{line:choose-n},  a minimal integer $n$ is chosen for which there exists a pair $(a,b)$ so that $2^n - c_\ell \cdot \ell = a^2 + 4b^2$. Under Heuristic \ref{heuristic:suitable-n}, we may assume that $n = O(\log(\ell))$.
	
	In the main part of the algorithm, the modular polynomial  $\varphi_\ell \in \FF_p[X,Y]$ is computed for various suitable primes $p \in P_\ell$, until the modulos $P$ is at least $\exp(B)$. At each step, we update $P = P \cdot p$ and compute  $\Phi_\ell \in \ZZ/P\ZZ$ using an explicit version of the Chinese Remainder Theorem. 
	
	As per Theorems \ref{thm:modular-polynomial-a} and \ref{thm:modular-polynomial-b}, the computation of $\varphi_\ell \in \FF_p[X,Y]$ is done in time \[
	O(\log p + \ell^2) \cdot \mathsf{M}(\FF_{p^2}) +  O(n\ell)  \cdot \mathsf{M}(R)
	+O(\log \ell) \mathsf{M}(R, \ell),
	\]
	with $R = \FF_{p^2}(\epsilon)/(\epsilon^{\ell+2})$.
	Further, we may assume that  $n = O(\log(\ell))$ under Heuristic \ref{heuristic:suitable-n}. Note that $B = O(\ell\log(\ell))$, hence it suffices to compute $\varphi_\ell$ for $O(p)$ many primes of size $\log(p) \in O(\log(\ell))$. Lemma \ref{lem:primes-in-Pl} assures that there are enough primes of the desired form in $\P_\ell$.  In this setting, 
	\begin{gather*}
	  \mathsf{M}(\FF_{p^2}) = \mathsf{M}(\log \ell)= O(\log\ell \log\log\ell \log\log\log \ell),\\
	  \mathsf{M}(R) = \mathsf{M}(\ell \log \ell) = O(\ell\log^2\ell\log\log\ell),\\
	  \mathsf{M}(R, \ell)=\mathsf{M}(\ell^2 \log \ell)=O(\ell^2 \log^2 \ell \log \log \ell)
    \end{gather*}
	 In conclusion, we obtain 
	 \[
	O(\ell^3 \log^3\ell \log\log\ell)
	 \]
	 for the overall runtime of Algorithm \ref{algo:full}.
\end{proof}

\begin{theorem} \label{thm:modular-polynomial2}
	There exists an algorithm for computing the modular polynomial $\Phi_\ell$ for any prime $\ell$, with unconditional runtime 
	\[
	O(\ell^3 \log^3\ell \log\log\ell).
	\]
\end{theorem}

\begin{proof}
We use the same proof as in \cref{thm:modular-polynomial}, with the exception that $P_\ell$ is replaced by the set of primes $\P^*_\ell$ defined in \cref{subsec:generalcase}. This requires to use dimension~$8$ embeddings instead of dimension~$2$ embeddings.
However, the advantage of $\P^*_\ell$ is that we can set
$n=\lceil\log_2(\ell)\rceil$
in the statement of \cref{lem:primes-in-Pl} without heuristic assumptions.
\end{proof}

\begin{remark}
	The runtime of the algorithms from Theorem \ref{thm:modular-polynomial} and \ref{thm:modular-polynomial2} improves to 
	\(
	O(\ell^3 \log^3 \ell),
	\)
	when the bound $\mathsf{M}(n) = O(n\log n)$ from \cite{harvey2021integer} is applied to describe the multiplication of $n$-bit integers, see also Subsection \ref{subsec:arithmetic-in-R}. 
\end{remark}

%% file: implementation.tex
\section{Implementation}

A proof-of-concept implementation of Algorithm \ref{algo:full} (with $\ell \equiv 3 \pmod{4}$) in SageMath \cite{sagemath} is available in our GitHub repository \cite{github-repo}. The repository further contains the different subroutines for computing with deformations of elliptic curves and isogenies presented in this paper. 

Our implementation works with elliptic curves in Montgomery form, and we use the available functions in SageMath to compute elliptic curve isogenies. The computation of $(2^n,2^n)$-isogenies is based on the formulas from \cite{kunzweiler2022efficient}, enhanced by explicit formulas for splitting and gluing isogenies. The individual $(2,2)$-isogenies in this framework are naturally represented by radical formulas which facilitates an efficient computation of the deformation of the isogeny chains. 

In the future, we plan to switch to the faster theta formulas
from \cite{DRfastthetadim2} to compute the $(2^n,2^n)$-isogenies and then use  radical $(2,2)$-isogenies in the theta model for computing the deformation of these isogenies.
However, to achieve a fast running time, we also require a dedicated implementation with a fast arithmetic for $R$ and fast polynomial multiplication over $R$, and that remains a future work.

A last remark is that we only require to compute $\ell$-isogenies in
dimension~$1$ for the initialisation step: the lifting step is done through $2^n$-isogenies in dimension~$2$.
We have seen that the initialisation step is not dominant, even if we use the Vélu formula rather than the sqrt-Vélu algorithm.
Still, we remark that one could also do the initialisation step using only
$2^n$-isogenies in dimension~$2$, using the Clapotis framework
\cite{DRclapotis} to convert reduced ideals of norm~$\ell$ to isogenies, as
is done in \cite{DRsqisign2d}. As a fun side effect, this would relax the
condition that $\ell \mid p^2 \pm 1$, and would allow to compute
$\phi_{\ell}$ while bypassing entirely Vélu's formulas for dimension~$1$
$\ell$-isogenies.

%% file: variants.tex
\section{Computing $\varphi_\ell \in \FF_p[x,y]$, when $\ell \equiv 1 \pmod{4}$}
\label{sec:variants}

Here, we present Algorithm \ref{algo:modp_3}, a variant of Algorithm \ref{algo:modp}, for computing the modular polynomial $\varphi_\ell$ over $\FF_p$ for $p \in \P_\ell$ and $\ell \equiv 1 \pmod{4}$.  These modifications were outlined in Subsection \ref{subsec:modular-poly-b}. In particular, we obtain the following explicit version of Theorem \ref{thm:modular-polynomial-b}.

\begin{theorem}\label{thm:modular-polynomial-b-2}
	On input a prime  $\ell \equiv 1 \pmod{4}$ and a prime $p \in \mathcal{P}_\ell$, Algorithm \ref{algo:modp_3} returns the modular polynomial $\varphi_\ell$ over $\FF_p$. The algorithm runs in time $O\left(\log(p) + \ell^2\log^2(\ell) + n\,\ell^2\log(\ell)\right)$ over $\FF_{p^2}$.
\end{theorem}

\begin{proof}
	This coincides with the proof of Theorem \ref{thm:modular-polynomial-b}.
\end{proof}

\begin{algorithm}[h!]
	\caption{\texttt{modular\_polynomial\_modp} ~ $\ell \equiv 1 \pmod{4}$ }\label{algo:modp_3}
	\begin{flushleft}
		\textbf{Input:} A prime  $\ell \equiv 1 \pmod{4}$, and a prime $p\in \mathcal{P}_\ell$.\\
		\textbf{Output:} The modular polynomial $\varphi_\ell(X,Y) \in \FF_p[X,Y]$.
	\end{flushleft}
	\begin{algorithmic}[1]
		\linecomment{Setup}
		\State Set $E: y^2 = x^3 + 6x^2 + x$ over $\FF_{p^2}$ 
		\State $\iota \gets \iota \in \End(E)$ with $\iota \circ {\iota} = [-4]$
		\State $\gamma \gets [a] + [b] \iota$, where $2^n - 3\cdot \ell = a^2 + 4b^2$ for some $n$. \label{line:definition-n}
		\State Sample $P_3 \in E[3] \setminus \{\O\}$ \label{line:aux-iso1}
		\State $C_0 \gets E /\langle P_3 \rangle$, $C_1 \gets E/\langle \gamma(P_3)\rangle$
		\State $\gamma_0, g, g' \gets$ isogenies with $\gamma_0 \circ g = g' \circ \gamma$ as in Eq. \ref{eq:large-isogeny-diamond}
		\label{line:aux-iso2}
		\State Compute $P_{2^n}, Q_{2^n}$ with $C_1[2^n] = \langle P_{2^n}, Q_{2^n} \rangle$ \label{line:torsion-C1}
		\State Compute $P_\ell, Q_\ell$ with $E[\ell] = \langle P_\ell, Q_\ell \rangle$
		\State $\tilde{j} \gets j(E) + \epsilon \in \FF_{p^2}[\epsilon]/(\epsilon^{\ell+2})$
		\State $\E \gets $ elliptic curve with $j(\E) = \tilde{j}$
		\State $\C_0 \gets $ codomain of the lift $\tilde{g}: \E \to \C_0$ \label{line:lift-C0}
		\linecomment{computing and lifting all $\ell$-isogenies}
		\For{$k \gets 0, \dots, \ell$}
		\If{$k = \ell$}
		\State $P_k \gets Q_\ell$
		\Else
		\State $P_k \gets P_\ell + k \cdot Q_\ell$
		\EndIf
		\State $E_k \gets E/\langle P_k\rangle$ with $f_k: E \to E_k$
		\linecomment{constructing the $(3\ell,2^n-3\ell)$-isogeny diamond}
		\State $E_k' \gets E/\langle \gamma(P_k)\rangle$ with $f_k': E \to E_k'$
		\State $K \gets \langle (\hat{\gamma_0}(P_{2^n}), f_k'\circ \hat{g'}(P_{2^n})), (\hat{\gamma_0}(Q_{2^n}), f_k'\circ \hat{g'}(Q_{2^n})) \rangle$
		\State $(\C_0,\E_k,\C_1,\E_k') \gets $ \texttt{lift\_isogeny\_diamond}$(C_0,E_k,C_1,E_k',K, \C_0)$
		\State $\tilde{j_k} \gets j(\E_k)$
		\EndFor
		\linecomment{final step}
		\State $\varphi \gets \prod_{k=0}^{\ell}(Y - \tilde{j_k})(\epsilon = X - j(E)) \in \FF_p[X,Y]$
		\State \Return $\varphi$
	\end{algorithmic}
\end{algorithm}

%% file: main.bbl
\begin{thebibliography}{10}

\bibitem{atkin1988number}
Arthur~OL Atkin.
\newblock The number of points on an elliptic curve modulo a prime.
\newblock {\em preprint}, 1988.

\bibitem{DRsqisign2d}
Andrea Basso, Luca~De Feo, Pierrick Dartois, Antonin Leroux, Luciano Maino,
  Giacomo Pope, Damien Robert, and Benjamin Wesolowski.
\newblock {SQIsign2D-West}: The fast, the small, and the safer.
\newblock Cryptology ePrint Archive, Paper 2024/760, 2024.

\bibitem{festa}
Andrea Basso, Luciano Maino, and Giacomo Pope.
\newblock {FESTA:} fast encryption from supersingular torsion attacks.
\newblock In Jian Guo and Ron Steinfeld, editors, {\em {ASIACRYPT} 2023, Part
  {VII}}, volume 14444 of {\em Lecture Notes in Computer Science}, pages
  98--126. Springer, 2023.

\bibitem{bernstein2007modular}
Daniel Bernstein and Jonathan Sorenson.
\newblock Modular exponentiation via the explicit chinese remainder theorem.
\newblock {\em Mathematics of Computation}, 76(257):443--454, 2007.

\bibitem{bernstein2020faster}
Daniel~J Bernstein, Luca De~Feo, Antonin Leroux, and Benjamin Smith.
\newblock Faster computation of isogenies of large prime degree.
\newblock {\em Open Book Series}, 4(1):39--55, 2020.

\bibitem{breuer2023explicit}
Florian Breuer, Desir{\'e}e~Gij{\'o}n G{\'o}mez, and Fabien Pazuki.
\newblock Explicit bounds on the coefficients of the modular polynomials and
  the size of $ {X}\_0 (n) $.
\newblock arXiv preprint arXiv:2310.14428, 2023.

\bibitem{broker2009constructing}
Reinier Br{\"o}ker.
\newblock Constructing supersingular elliptic curves.
\newblock {\em J. Comb. Number Theory}, 1(3):269--273, 2009.

\bibitem{broker2012modular}
Reinier Br{\"o}ker, Kristin Lauter, and Andrew Sutherland.
\newblock Modular polynomials via isogeny volcanoes.
\newblock {\em Mathematics of Computation}, 81(278):1201--1231, 2012.

\bibitem{broker2010explicit}
Reinier Br{\"o}ker and Andrew~V Sutherland.
\newblock An explicit height bound for the classical modular polynomial.
\newblock {\em The Ramanujan Journal}, 22:293--313, 2010.

\bibitem{castryck2021multiradical}
Wouter Castryck and Thomas Decru.
\newblock Multiradical isogenies.
\newblock {\em Arithmetic, Geometry, Cryptography, and Coding Theory},
  779:57--89, 2021.

\bibitem{castryck2023efficient}
Wouter Castryck and Thomas Decru.
\newblock An efficient key recovery attack on {SIDH}.
\newblock In Carmit Hazay and Martijn Stam, editors, {\em {EUROCRYPT} 2023,
  Part {V}}, volume 14008 of {\em Lecture Notes in Computer Science}, pages
  423--447. Springer, 2023.

\bibitem{costello2017efficient}
Craig Costello, David Jao, Patrick Longa, Michael Naehrig, Joost Renes, and
  David Urbanik.
\newblock Efficient compression of {SIDH} public keys.
\newblock In {\em EUROCRYPT 2017, Part I 36}, pages 679--706. Springer, 2017.

\bibitem{sqisignhd}
Pierrick Dartois, Antonin Leroux, Damien Robert, and Benjamin Wesolowski.
\newblock {SQISignHD}: New dimensions in cryptography.
\newblock In Marc Joye and Gregor Leander, editors, {\em {EUROCRYPT} 2024, Part
  {I}}, volume 14651 of {\em Lecture Notes in Computer Science}, pages 3--32.
  Springer, 2024.

\bibitem{DRfastthetadim2}
Pierrick Dartois, Luciano Maino, Giacomo Pope, and Damien Robert.
\newblock An algorithmic approach to $(2,2)$-isogenies in the theta model and
  applications to isogeny-based cryptography.
\newblock Cryptology ePrint Archive, Paper 2023/1747, 2023.

\bibitem{deligne_SchemasModulesCourbes1973}
P.~Deligne and M.~Rapoport.
\newblock Les sch\'emas de modules de courbes elliptiques.
\newblock In {\em Modular functions of one variable, {{II}} ({{Proc}}.
  {{Internat}}. {{Summer School}}, {{Univ}}. {{Antwerp}}, {{Antwerp}}, 1972)},
  pages 143--316. Lecture Notes in Math., Vol. 349, 1973.

\bibitem{dupont2006moyenne}
R.~Dupont.
\newblock {Moyenne arithmetico-geometrique, suites de Borchardt et
  applications}.
\newblock {\em These de doctorat, Ecole polytechnique, Palaiseau}, 2006.

\bibitem{elkies1991explicit}
Noam~D Elkies.
\newblock Explicit isogenies.
\newblock {\em preprint}, 1991.

\bibitem{enge2009computing}
Andreas Enge.
\newblock Computing modular polynomials in quasi-linear time.
\newblock {\em Mathematics of Computation}, 78(267):1809--1824, 2009.

\bibitem{hartshorne2010deformation}
Robin Hartshorne.
\newblock {\em Deformation theory}, volume 257.
\newblock Springer, 2010.

\bibitem{harvey2021integer}
David Harvey and Joris Van Der~Hoeven.
\newblock Integer multiplication in time o(nlog$\backslash$,n).
\newblock {\em Annals of Mathematics}, 193(2):563--617, 2021.

\bibitem{iwaniec2021analytic}
Henryk Iwaniec and Emmanuel Kowalski.
\newblock {\em Analytic number theory}, volume~53.
\newblock American Mathematical Soc., 2021.

\bibitem{kani1997number}
Ernst Kani.
\newblock The number of curves of genus two with elliptic differentials.
\newblock {\em J. reine angew. Math. 485}, pages 93--121, 1997.

\bibitem{kieffer2020evaluating}
Jean Kieffer.
\newblock Evaluating modular polynomials in genus 2.
\newblock arXiv preprint arXiv:2010.10094, 2020.

\bibitem{kieffer2020sign}
Jean Kieffer.
\newblock Sign choices in the {AGM} for genus two theta constants.
\newblock {\em Publications math{\'e}matiques de Besan{\c c}on. Alg{\`e}bre et
  th{\'e}orie des nombres}, pages 37--58, 2022.

\bibitem{kieffer2022certified}
Jean Kieffer.
\newblock Certified newton schemes for the evaluation of low-genus theta
  functions.
\newblock {\em Numerical Algorithms}, 93(2):833--862, 2023.

\bibitem{kronecker1882grundzuge}
L.~Kronecker.
\newblock Grundzüge einer arithmetischen {T}heorie der algebraischen
  {G}rössen. ({A}bdruck einer {F}estschrift zu {H}errn {E. E. K}ummers
  {D}octor-{J}ubiläum, 10. {S}eptember 1881.).
\newblock {\em {J}ournal für die reine und angewandte {M}athematik},
  92:1--122, 1882.

\bibitem{kunzweiler2022efficient}
Sabrina Kunzweiler.
\newblock Efficient computation of $(2^n,2^n)$-isogenies.
\newblock {\em Designs, Codes and Cryptography}, 92(6):1761--1802, 2024.

\bibitem{github-repo}
Sabrina Kunzweiler and Damien Robert.
\newblock Computing modular polynomials by deformation.
\newblock \url{https://github.com/sabrinakunzweiler/modular-polynomials}, 2024.

\bibitem{leroux2023computation}
Antonin Leroux.
\newblock Computation of {H}ilbert class polynomials and modular polynomials
  from supersingular elliptic curves.
\newblock arXiv preprint arXiv:2301.08531, 2023.

\bibitem{DRfastisogenies}
David Lubicz and Damien Robert.
\newblock Fast change of level and applications to isogenies.
\newblock {\em Research in Number Theory (ANTS XV Conference)}, 9(1), 2022.

\bibitem{eurocrypt-2023-32955}
Luciano Maino, Chloe Martindale, Lorenz Panny, Giacomo Pope, and Benjamin
  Wesolowski.
\newblock A direct key recovery attack on {SIDH}.
\newblock In {\em {EUROCRYPT} 2023, Part {V}}, volume 14008 of {\em Lecture
  Notes in Computer Science}, pages 448--471. Springer, 2023.

\bibitem{milio2015quasi}
Enea Milio.
\newblock A quasi-linear time algorithm for computing modular polynomials in
  dimension 2.
\newblock {\em LMS Journal of Computation and Mathematics}, 18(1):603--632,
  2015.

\bibitem{MumfordOEDAV2}
David Mumford.
\newblock On the equations defining abelian varieties. {II}.
\newblock {\em Invent. Math.}, 3:75--135, 1967.

\bibitem{oort1971finite}
Frans Oort.
\newblock Finite group schemes, local moduli for abelian varieties, and lifting
  problems.
\newblock {\em Compositio Mathematica}, 23(3):265--296, 1971.

\bibitem{DRclapotis}
Aurel Page and Damien Robert.
\newblock Introducing clapoti(s): Evaluating the isogeny class group action in
  polynomial time.
\newblock Cryptology ePrint Archive, Paper 2023/1766, 2023.

\bibitem{hdr}
Damien Robert.
\newblock {\em Efficient algorithms for abelian varieties and their moduli
  spaces}.
\newblock Habilitation à diriger des recherches, Universit{\'e} de Bordeaux
  (UB), 2021.

\bibitem{robert2022evaluating}
Damien Robert.
\newblock Evaluating isogenies in polylogarithmic time.
\newblock Cryptology ePrint Archive, Paper 2022/1068, 2022.

\bibitem{robert2022applications}
Damien Robert.
\newblock Some applications of higher dimensional isogenies to elliptic curves
  (overview of results).
\newblock Cryptology ePrint Archive, Paper 2022/1704, 2022.

\bibitem{EC23:Robert}
Damien Robert.
\newblock Breaking {SIDH} in polynomial time.
\newblock In Carmit Hazay and Martijn Stam, editors, {\em {EUROCRYPT} 2023,
  Part {V}}, volume 14008 of {\em Lecture Notes in Computer Science}, pages
  472--503. Springer, 2023.

\bibitem{schoof1995counting}
Ren{\'e} Schoof.
\newblock Counting points on elliptic curves over finite fields.
\newblock {\em Journal de th{\'e}orie des nombres de Bordeaux}, 7(1):219--254,
  1995.

\bibitem{schost2003computing}
{\'E}ric Schost.
\newblock Computing parametric geometric resolutions.
\newblock {\em Applicable Algebra in Engineering, Communication and Computing},
  13(5):349--393, 2003.

\bibitem{sernesi2007deformations}
Edoardo Sernesi.
\newblock {\em Deformations of algebraic schemes}, volume 334.
\newblock Springer Science \& Business Media, 2007.

\bibitem{stacks-project}
The {Stacks Project Authors}.
\newblock \textit{Stacks Project}.
\newblock \url{https://stacks.math.columbia.edu}, 2018.

\bibitem{tate1967p}
John~T Tate.
\newblock p-divisible groups.
\newblock In {\em Proceedings of a Conference on Local Fields: NUFFIC Summer
  School held at Driebergen (The Netherlands) in 1966}, pages 158--183.
  Springer, 1967.

\bibitem{sagemath}
{The Sage Developers}.
\newblock {\em {S}ageMath, the {S}age {M}athematics {S}oftware {S}ystem
  ({V}ersion 10.0)}, 2024.
\newblock {\tt https://www.sagemath.org}.

\bibitem{velu1971isogenies}
Jacques V{\'e}lu.
\newblock Isog{\'e}nies entre courbes elliptiques.
\newblock {\em Comptes-Rendus de l'Acad{\'e}mie des Sciences}, 273:238--241,
  1971.

\bibitem{wesolowski2022supersingular}
Benjamin Wesolowski.
\newblock The supersingular isogeny path and endomorphism ring problems are
  equivalent.
\newblock In {\em 2021 IEEE 62nd Annual Symposium on Foundations of Computer
  Science (FOCS)}, pages 1100--1111. IEEE, 2022.

\end{thebibliography}
